\documentclass[]{amsproc}

\usepackage{amsmath, amsthm, amssymb, mathtools, mathrsfs, stmaryrd}
\usepackage{ascmac}
\usepackage{comment}
\usepackage{bm, bbm}
\allowdisplaybreaks

\usepackage{graphicx}
\usepackage[top=30mm, bottom=30mm, left=29mm, right=30mm]{geometry}

\usepackage{hyperref}

\usepackage{tikz}
\usetikzlibrary{intersections, calc, arrows.meta}

\usepackage{here}
\usepackage{time}
\usepackage[abbrev]{amsrefs}

\usepackage{xcolor}
\usepackage[capitalize,nameinlink,noabbrev,nosort]{cleveref}
\hypersetup{
	colorlinks=true,       
	linkcolor=brown,          
	citecolor=brown,        
	filecolor=brown,      
	urlcolor=brown,           
}

\makeatletter
\@namedef{subjclassname@2020}{\textup{2020} Mathematics Subject Classification}
\makeatother

\newtheorem*{theorem*}{\hspace{-6.3mm}\textbf{Theorem}}  

\newtheorem{theoremcounter}{Theorem Counter}[section]

\theoremstyle{remark}
\newtheorem{remark}{Remark}

\theoremstyle{definition}
\newtheorem{definition}[theoremcounter]{Definition}
\newtheorem{example}{Example}

\theoremstyle{plain}
\newtheorem{lemma}[theoremcounter]{Lemma}

\newtheorem{corollary}[theoremcounter]{Corollary}

\newtheorem{theorem}[theoremcounter]{Theorem}

\numberwithin{equation}{section}

\newcommand{\Z}{\mathbb{Z}}
\newcommand{\Q}{\mathbb{Q}}
\newcommand{\R}{\mathbb{R}}
\newcommand{\C}{\mathbb{C}}

\newcommand{\dd}{\mathrm{d}}
\newcommand{\bbH}{\mathbb{H}}

\DeclareMathOperator{\ImNew}{Im}
\renewcommand{\Im}{\ImNew}
\DeclareMathOperator{\ReNew}{Re}
\renewcommand{\Re}{\ReNew}

\DeclareMathOperator{\SL}{SL}
\DeclareMathOperator{\PSL}{PSL}

\DeclareMathOperator{\sgn}{sgn}

\newcommand{\pmat}[1]{\begin{pmatrix}#1\end{pmatrix}}

\newcommand{\smat}[1]{\bigl(\begin{smallmatrix}#1\end{smallmatrix}\bigr)}

\begin{document}

\title[]{A unified approach to Rohrlich-type divisor sums}

\author{Daeyeol Jeon}
\address{Department of Mathematics Education, Kongju National University, Kongju, 314-701, Republic of Korea}
\email{dyjeon@kongju.ac.kr}

\author{Soon-Yi Kang}
\address{Department of Mathematics, Kangwon National University, Chuncheon, 200-701, Republic of Korea}
\email{sy2kang@kangwon.ac.kr}

\author{Chang Heon Kim}
\address{Department of Mathematics, Sungkyunkwan University, Suwon, 16419, Republic of Korea}
\email{chhkim@skku.edu}

\author{Toshiki Matsusaka}
\address{Faculty of Mathematics, Kyushu University, Motooka 744, Nishi-ku, Fukuoka 819-0395, Japan}
\email{matsusaka@math.kyushu-u.ac.jp}

\thanks{The first author was supported by the Research year project of the Kongju National University in 2023, 
the second author was supported by the National Research Foundation of Korea (NRF) funded by the Ministry of Education (NRF-2022R1A2C1007188), and 
the third author was supported by the National Research Foundation of Korea(NRF) grant funded by the Korea government (MSIT) (RS-2024-00348504).
The fourth author was supported by the MEXT Initiative through Kyushu University's Diversity and Super Global Training Program for Female and Young Faculty (SENTAN-Q) and JSPS KAKENHI (JP21K18141 and JP24K16901).}

\subjclass[2020]{Primary 11F25; Secondary 11F37}



\maketitle

\begin{abstract}
	We propose a systematic method for analyzing Rohrlich-type divisor sums for arbitrary congruence subgroups $\Gamma_0(N)$. Our main theorem unifies various results from the literature, and its significance is illustrated through the following five applications: (1) the valence formula, (2) a natural generalization of classical Rohrlich's formula to level $N$, (3) an explicit version of the theorem by Bringmann--Kane--L\"{o}brich--Ono--Rolen, (4) an extension of the generalized Rohrlich formula proposed by Bringmann--Kane, and (5) an alternative proof of the decomposition formula for twisted traces of CM values of weight 0 Eisenstein series.
\end{abstract}


\section{Introduction}

Let $j(\tau)$ denote the classical elliptic modular invariant, a holomorphic function on the upper-half plane $\bbH \coloneqq \{\tau = u+iv\in \C : u,v \in \R, v > 0\}$ that is invariant under the action of $\SL_2(\Z)$ and has a simple pole at the cusp $i\infty$. For any positive integer $n$, let $j_n(\tau)$ be the $n$-th element of the Hecke system studied by Asai--Kaneko--Ninomiya~\cite{AsaiKanekoNinomiya1997}, which is a polynomial in $j(\tau)$ normalized so that the non-positive power terms in its Fourier expansion is $q^{-n} + 24 \sigma_1(n)$ with $q \coloneqq e^{2\pi i\tau}$. The first few terms are given by
\begin{align*}
	j_1(\tau) &= j(\tau) - 720 = q^{-1} + 24 + 196884 q + 21493760 q^2 + \cdots,\\
	j_2(\tau) &= j(\tau)^2 - 1488 j(\tau) + 159840 = q^{-2} + 72 + 42987520 q + 40491909396 q^2 + \cdots.
\end{align*}
More generally, for an integer $k \in \Z$, a meromorphic function on $\bbH$ and at all cusps that satisfies $(f|_k\gamma)(\tau) \coloneqq (c\tau+d)^{-k} f(\gamma \tau) = f(\tau)$ for any $\gamma = \smat{a & b \\ c & d} \in \SL_2(\Z)$ is called a meromorphic modular form of weight $k$. The following result is established by Bruinier--Kohnen--Ono~\cite[Theorem 1]{BruinierKohnenOno2004}.

\begin{theorem*}
	Let $\Theta \coloneqq \frac{1}{2\pi i} \frac{\dd}{\dd \tau}$. For a positive integer $n > 0$ and a weight $k$ meromorphic modular form $f$ on $\SL_2(\Z)$ that has neither a zero nor a pole at the cusp $i\infty$, we have
	\[
		\sum_{z \in \SL_2(\Z) \backslash \bbH} \frac{\mathrm{ord}_z(f)}{\# \PSL_2(\Z)_z} j_n(z) = -\mathrm{Coeff}_{q^n} \left(\frac{\Theta f}{f}\right),
	\]
	where $\mathrm{ord}_z(f)$ is the order of $f$ at $\tau = z$ and $\mathrm{Coeff}_{q^n}(f)$ denotes the $n$-th Fourier coefficient of $f(\tau)$.
\end{theorem*}

As they point out, this theorem offers a valuable connection between the values of $j$-function and the arithmetic properties of the Fourier coefficients of modular forms, suggesting various significant implications. In this context, the sums of special values on divisors of meromorphic modular forms have been explored in numerous situations beyond the $j$-function. For instance, replacing $j_n(\tau)$ with the constant function 1 leads to the valence formula, which is instrumental in determining the dimension of the space of modular forms. In 1984, Rohrlich~\cite{Rohrlich1984} examined the divisor sums of the Kronecker limit function $\log(v^6 |\Delta(\tau)|)$, where $\Delta(\tau) \coloneqq q \prod_{n=1}^\infty (1-q^n)^{24}$, in search of analogies to Jensen's formula from complex analysis within the framework of modular forms. In reference to this classical work of Rohrlich, we will refer to sums of this form as Rohrlich-type divisor sums. More recently, this work has been generalized by Bringmann--Kane~\cite{BringmannKane2020} to sesqui-harmonic Maass forms associated with the $j$-function.

The aim of this article is to establish a fundamental theorem that unifies the proof of these work on Rohrlich-type divisor sums, while also generalizing and refining them for arbitrary congruence subgroups $\Gamma_0(N)$. To state the claim, we first introduce some notations. Let $\mathcal{C}_N \coloneqq \Gamma_0(N) \backslash (\Q \cup \{i\infty\})$ be the set of cusps of $\Gamma_0(N)$. For each cusp $\rho$, let $\sigma_\rho \in \SL_2(\Z)$ be a scaling matrix and $\ell_\rho$ be its width as defined in \cref{sec:punctured}. For an integer $k \in \Z$ and a smooth function $f: \bbH \to \C$, we define differential operators
\begin{align*}
	\mathcal{D}_k f (\tau) &\coloneqq \frac{\Theta f(\tau)}{f(\tau)} - \frac{k}{4\pi v},\\
	\xi_k f(\tau) &\coloneqq 2i v^k \overline{\frac{\partial}{\partial \overline{\tau}} f(\tau)}, 
\end{align*}
and $\Delta_k \coloneqq -\xi_{2-k} \circ \xi_k$. It is known that if $f$ satisfies $f|_k \gamma = f$ for $\gamma \in \SL_2(\Z)$, then $(\mathcal{D}_k f)|_2 \gamma = \mathcal{D}_k f$ and $(\xi_k f)|_{2-k} \gamma = \xi_k f$ hold, (see~\cite[Chapter 5]{BFOR2017}). Furthermore, for two weight 0 modular forms $f$ and $g$, let $\langle f, g \rangle^{\mathrm{reg}}$ denote the regularized Petersson inner product defined in \cref{def:reg-inner}. The following is our main theorem.

\begin{theorem}\label{thm:main}
For a meromorphic modular form $f$ of weight $k$ and a smooth modular form $g$ of weight $2$ (without singularities on $\bbH$) on $\Gamma_0(N)$, we have that
\begin{align*}
	\langle \xi_2 \Delta_2 g, &\log(v^{k/2} |f|) \rangle^\mathrm{reg} = -2\pi \sum_{z \in \Gamma_0(N) \backslash \bbH} \frac{\mathrm{ord}_z(f)}{\omega_z} \xi_2 g(z) + \sum_{\rho \in \mathcal{C}_N} \bigg(C_{1,\rho}(f,g) + \overline{C_{2,\rho}(f,g)} \bigg),
\end{align*}
if the limits $C_{1, \rho}(f,g), C_{2, \rho}(f,g)$ defined by
\begin{align*}
	C_{1,\rho} (f,g) &\coloneqq -2\pi \lim_{y \to \infty} \int_{iy}^{\ell_\rho+iy} (\mathcal{D}_kf)|_2 \sigma_\rho(\tau) \cdot (\xi_2 g)|_0 \sigma_\rho (\tau) \dd \tau,\\
	C_{2,\rho} (f,g) &\coloneqq \lim_{y \to \infty} \int_{iy}^{\ell_\rho+iy} \left(\log(v^{k/2}|(f|_k \sigma_\rho)(\tau)|) \cdot (\Delta_2 g)|_2\sigma_\rho (\tau) + \frac{k}{2} (g|_2\sigma_\rho) (\tau) \right) \dd \tau
\end{align*}
exist. Here, we set $\omega_z \coloneqq \#\overline{\Gamma_0(N)}_z$ and $\overline{\Gamma_0(N)} \coloneqq \Gamma_0(N)/\{\pm I\}$.
\end{theorem}

The significance lies not in the theorem itself, but in its consequences. As proven in \cref{sec:Reg-inner} and \cref{sec:proof}, this theorem is established using Stokes' theorem. In \cref{sec:Applications}, we present five applications. As the simplest of these, in \cref{sec:cor1}, we show that the valence formula follows by taking the weight 2 Eisenstein series as $g(\tau)$. In \cref{sec:Rohrlich}, we derive a natural generalization of Rohrlich's formula to level $N$ using a polyharmonic Maass form based on the real analytic Eisenstein series $E_{N,2}(\tau,s)$. In \cref{sec:cor3}, we extend the aforementioned result by Bruinier--Kohnen--Ono to level $N$ through a sesqui-harmonic Maass form derived from the Maass--Poincar\'{e} series. Although this extension has been addressed by Bringmann--Kane--L\"{o}brich--Ono--Rolen~\cite{BKLOR2018}, their Theorem 1.3 only refers to the claim in a form that allows for an error concerning the space of cusp forms, whereas we provide a more explicit equality. In \cref{sec:cor4}, we extend Bringmann--Kane's generalized Rohrlich formula to level $N$. Their proof relies on the fact that $\SL_2(\Z)$ has only one cusp. However, we present a result for general $\Gamma_0(N)$ that does not depend on the number of cusps and its genus. Finally, in \cref{sec:cor5}, we apply our theorem to the Eisenstein series $E_{1,2}(\tau,s)$ with complex variable $s$. By combining results on generalized Borcherds products, we offer a more concise alternative proof of the decomposition formula for the twisted traces of CM values of the weight 0 Eisenstein series, as given by Duke--Imamo\={g}lu--T\'{o}th~\cite{DukeImamogluToth2011} and Kaneko--Mizuno~\cite{KanekoMizuno2020}. Each of these earlier works will be explained in more detail in the corresponding sections.

\section{Regularized Petersson inner product}\label{sec:Reg-inner}

\subsection{Punctured fundamental domain}\label{sec:punctured}

To clarify the theorem's statement, we introduce the regularized Petersson inner product based on the definition by Bringmann--Kane~\cite[Section 6]{BringmannKane2020}. For the congruence subgroup $\Gamma_0(N) \coloneqq \{ \gamma = \smat{a & b \\ c & d} \in \SL_2(\Z) : c \equiv 0 \pmod{N}\}$ with $N \ge 1$, we let $\mathcal{F}_0^*(N)$ denote a closed fundamental domain for the action of $\Gamma_0(N)$ on $\bbH^* \coloneqq \bbH \cup \Q \cup \{i\infty\}$ via the M\"{o}bius transformation $\gamma \tau \coloneqq \frac{a\tau + b}{c\tau + d}$. For each cusp $\rho \in \mathbb{Q} \cup \{i\infty\}$, let $\sigma_\rho \in \SL_2(\Z)$ be a scaling matrix satisfying $\sigma_\rho i\infty = \rho$. Moreover, we can choose $\sigma_\rho$ such that $\sigma_\rho \sigma_{\rho'}^{-1} \in \Gamma_0(N)$ for any $\Gamma_0(N)$-equivalent pair $\rho, \rho'$. Then there exists $\ell_\rho \in \Z_{>0}$ such that the stabilizer subgroup $\Gamma_0(N)_\rho$ satisfies
\[
	\sigma_\rho^{-1} \Gamma_0(N)_\rho \sigma_\rho = \left\{\pm \pmat{1 & n\ell_\rho \\ 0 & 1} \in \SL_2(\Z) : n \in \Z\right\}.
\]
We refer to $\ell_\rho$ as the width of the cusp $\rho$. As is clear from the definition, we have $\ell_{i\infty} = 1$. We let $\{\rho_1, \dots, \rho_n\} = \mathcal{F}_0^*(N) \cap (\Q \cup \{i\infty\})$ be the set of cusps of $\mathcal{F}_0^*(N)$. It is important to note that distinct $\rho_j$'s may still be equivalent under the action of $\Gamma_0(N)$. For each cusp $\rho$ of $\mathcal{F}_0^*(N)$ and $y > 0$, we define an $y$-neighborhood $B_{\rho, y}$ as
\[
	B_{\rho, y} = \{\sigma_\rho \tau \in \bbH : \Im(\tau) > y\} \cap \mathcal{F}_0^*(N).
\]
In addition, for each point $z \in \bbH \cap \mathcal{F}_0^*(N)$ and $\varepsilon > 0$, we define an $\varepsilon$-neighborhood $B_{z, \varepsilon}$ by
\[
	B_{z, \varepsilon} \coloneqq \{\tau \in \bbH : \|\tau - z\| < \varepsilon\} \cap \mathcal{F}_0^*(N).
\]

For a finite subset $\{z_1, \dots, z_m\} \subset \bbH \cap \mathcal{F}_0^*(N)$, we define the punctured fundamental domain $\mathcal{F}_0^{\varepsilon, y} (N)$ by
\[
	\mathcal{F}_0^{\varepsilon, y} (N) \coloneqq \mathcal{F}_0^*(N) - \left(\bigcup_{j=1}^m B_{z_j, \varepsilon} \cup \bigcup_{j=1}^n B_{\rho_j, y} \right).
\]

\begin{example}
For the level 4 congruence subgroup $\Gamma_0(4)$, we can take a closed fundamental domain $\mathcal{F}_0^*(4)$ as described below. The set of its cusps is given by $\{i\infty, 0, 1/2, 1\}$, where $0$ and $1$ are $\Gamma_0(4)$-equivalent, while any other pair is not. We can take $\sigma_{i\infty} = I, \sigma_0 = \smat{0 & -1 \\ 1 & 0}, \sigma_{1/2} = \smat{1 & -1 \\ 2 & -1}, \sigma_1 = \smat{1 & -1 \\ 1 & 0}$ as the scaling matrices, with the cusp widths $\ell_{i\infty} = 1, \ell_0 = \ell_1 = 4, \ell_{1/2} = 1$. Note that we choose $\sigma_1 \sigma_0^{-1} = \smat{1 & 1 \\0 & 1} \in \Gamma_0(4)$.
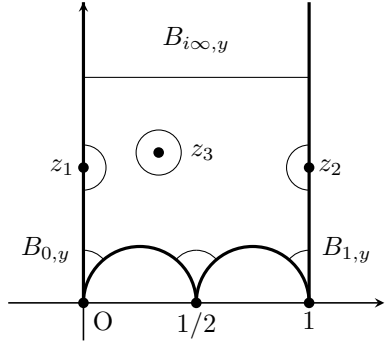
\begin{figure}[H]
\begin{center}
\begin{tikzpicture}
	\draw[->,>=stealth,semithick] (-1,0)--(4,0); 
	\draw[->,>=stealth,semithick] (0,-0.5)--(0,4); 
	\draw (0,0)node[below right]{O}; 
	\draw [very thick] (0,0)--(0,4); \draw [very thick] (3,0)--(3,4); 
	\draw [very thick] (0,0) arc (180:0:0.75); \draw [very thick] (1.5,0) arc (180:0:0.75);
	\draw (1.5,0) node[below]{$1/2$};
	\draw (3,0) node[below]{$1$};
	\fill (0,0) circle[radius=2pt]; \fill (1.5,0) circle[radius=2pt]; \fill (3,0) circle[radius=2pt]; 
	\draw (0,3)--(3,3); \draw (1.5,3.5) node{$B_{i\infty, y}$};
	\draw (0,0.7) arc (90:39:0.35); \draw (-0.5,0.7) node{$B_{0, y}$}; 
	\draw (3.5,0.7) node{$B_{1,y}$};
	\draw (1.5,0.7) arc (90:39:0.35); \draw (1.5,0.7) arc (90:141:0.35); \draw (3,0.7) arc (90:141:0.35);
	\draw (0, 1.5) arc (-90:90:0.3); \draw (3, 1.5) arc (270:90:0.3); 
	\fill (0,1.8) circle[radius=2pt]; \fill (3,1.8) circle[radius=2pt]; \draw (1,2) circle [radius=0.3]; \fill (1,2) circle[radius=2pt]; \draw (0, 1.8) node[left]{$z_1$}; \draw (3, 1.8) node[right]{$z_2$}; \draw (1.3,2) node[right]{$z_3$};
\end{tikzpicture}
\end{center}
\caption{A fundamental domain $\mathcal{F}_0^*(4)$ and neighborhoods for a finite set $\{z_1, z_2, z_3\}$.}
\label{fig:fundamental-domain}
\end{figure}
Through a direct calculation, we can explicitly express the equations for the boundary arc  $C_\rho$ of a $y$-neighborhood at each cusp $\rho$ as follows: $C_{i\infty}(t) = t+iy$ ($0 \le t \le 1$), $C_0(t) = \sigma_0 (t+iy)$ ($-2 \le t \le 0$) on the circle $X^2 + (Y- \frac{1}{2y})^2 = \frac{1}{4y^2}$, $C_{1/2}(t) = \sigma_{1/2} (t+iy)$ ($0 \le t \le 1$) on the circle $(X-\frac{1}{2})^2 + (Y- \frac{1}{8y})^2 = \frac{1}{64y^2}$, and $C_1(t) = \sigma_1 (t+iy)$ ($0 \le t \le 2$) on the circle $(X-1)^2 + (Y - \frac{1}{2y})^2 = \frac{1}{4y^2}$. The width of a cusp is equal to the sum of the lengths of the $t$-intervals corresponding to all equivalent cusps.
\end{example}

Under these notation, we can define the regularized inner product.

\begin{definition}\label{def:reg-inner}
	For two functions $f, g: \bbH \to \C$ satisfying $f(\gamma \tau) = f(\tau)$ and $g(\gamma \tau) = g(\tau)$ for any $\gamma \in \Gamma_0(N)$, we assume that the set of singularities $\{z_1, \dots, z_n\}$ of $f$ and $g$ on $\bbH \cap \mathcal{F}_0^*(N)$ is finite. Then, the regularized Petersson inner product is defined by
	\[
		\langle f, g \rangle^\mathrm{reg} \coloneqq \lim_{\varepsilon \to 0} \lim_{y \to \infty} \int_{\mathcal{F}_0^{\varepsilon, y}(N)} f(\tau) \overline{g(\tau)} \frac{\dd u \dd v}{v^2}
	\]
	if the limit exists.
\end{definition}

\subsection{Auxiliary calculations using Stokes' theorem}

Under the assumptions of \cref{thm:main}, the functions $\xi_2 \Delta_2 g(\tau)$ and $\log(v^{k/2} |f(\tau)|)$ are invariants under the action of $\Gamma_0(N)$. Furthermore, the set of singularities of these functions on $\bbH \cap \mathcal{F}_0^*(N)$ is a finite set consisting of the zeros and poles of $f$. Thus, we can consider their regularized Petersson inner product. To calculate it by using Stokes' theorem, we show the following lemma.

\begin{lemma}\label{lem:Stokes}
For a meromorphic modular form $f$ of weight $k$ and a smooth modular form $g$ of weight $2$ (without singularities on $\bbH$) on $\Gamma_0(N)$, we have
\begin{align*}
	&(\xi_2 \Delta_2 g(\tau)) \cdot \log(v^{k/2} |f(\tau)|) \frac{\dd u \wedge \dd v}{v^2}\\
	&=\dd \bigg( 2\pi \mathcal{D}_k f(\tau) \cdot \xi_2 g(\tau) \dd \tau + \left(-\log(v^{k/2}|f(\tau)|) \cdot \overline{\Delta_2 g(\tau)} - \frac{k}{2} \overline{g(\tau)} \right) \dd \overline{\tau} \bigg).
\end{align*}
\end{lemma}

\begin{proof}
By our assumption that $f$ is meromorphic, we have
\[
	\frac{\partial}{\partial \overline{\tau}} \mathcal{D}_k f(\tau) = - \frac{k}{4\pi} \cdot \frac{1}{2} \left(\frac{\partial}{\partial u} + i \frac{\partial}{\partial v}\right) \frac{1}{v} = - \frac{k}{4\pi} \frac{1}{2i} \frac{1}{v^2}.
\]
Since $\dd (f \dd \tau) = \frac{\partial}{\partial \overline{\tau}} f \dd \overline{\tau} \wedge \dd \tau = 2i \frac{\partial}{\partial \overline{\tau}} f \dd u \wedge \dd v$ and
\[
	2i \frac{\partial}{\partial \overline{\tau}} \xi_2 g(\tau) = \overline{-2i \overline{\frac{\partial}{\partial \overline{\tau}} \xi_2 g(\tau)}} = \overline{-\xi_0 \xi_2 g(\tau)} = \overline{\Delta_2 g(\tau)},
\]
the first term of the right-hand side is given as
\begin{align}\label{eq:lem-1}
	\dd \bigg(2\pi \mathcal{D}_k f(\tau) \cdot \xi_2 g(\tau) \dd \tau \bigg) &= \left(-\frac{k}{2} \frac{1}{v^2} \cdot \xi_2 g(\tau) + 2\pi \mathcal{D}_k f(\tau) \cdot \overline{\Delta_2 g(\tau)} \right) \dd u \wedge \dd v.
\end{align}
Similarly, for the second term, we apply $\dd(f \dd \overline{\tau}) = \frac{\partial}{\partial \tau} f \dd \tau \wedge \dd \overline{\tau} = -2i \frac{\partial}{\partial \tau} f \dd u \wedge \dd v$,
\begin{align}\label{eq:log-der}
	-2i \frac{\partial}{\partial \tau} \log(v^{k/2} |f(\tau)|) &= -2i \frac{\partial}{\partial \tau} \left(\frac{k}{2} \log v + \frac{1}{2} \log f(\tau) + \frac{1}{2} \log \overline{f(\tau)}\right) = 2\pi \mathcal{D}_k f(\tau),
\end{align}
and
\[
	-2i \frac{\partial}{\partial \tau} \overline{\Delta_2 g(\tau)} = -2i \overline{\frac{\partial}{\partial \overline{\tau}} \Delta_2 g(\tau)} = - \frac{1}{v^2} \xi_2 \Delta_2 g(\tau)
\]
to have
\begin{align}\label{eq:lem-2}
\begin{split}
	&\dd \left(-\log(v^{k/2}|f(\tau)|) \cdot \overline{\Delta_2 g(\tau)} - \frac{k}{2} \overline{g(\tau)} \right) \dd \overline{\tau}\\
	&= \left( -2\pi \mathcal{D}_k f(\tau) \cdot \overline{\Delta_2 g(\tau)} + \frac{1}{v^2} \log(v^{k/2}|f(\tau)|) \cdot \xi_2 \Delta_2 g(\tau) + \frac{k}{2} \frac{1}{v^2} \cdot \xi_2 g(\tau) \right) \dd u \wedge \dd v.
\end{split}
\end{align}
By summing \eqref{eq:lem-1} and \eqref{eq:lem-2}, we obtain the desired result.
\end{proof}

By \cref{lem:Stokes} and Stokes' theorem, we obtain
\begin{align}\label{eq:Stokes}
\begin{split}
	&\int_{\mathcal{F}_0^{\varepsilon,y}(N)} (\xi_2 \Delta_2 g(\tau)) \cdot \log(v^{k/2} |f(\tau)|)\frac{\dd u \dd v}{v^2}\\
	&= \int_{\partial \mathcal{F}_0^{\varepsilon,y}(N)} 2\pi \mathcal{D}_k f(\tau) \cdot \xi_2 g(\tau) \dd \tau + \left(-\log(v^{k/2}|f(\tau)|) \cdot \overline{\Delta_2 g(\tau)} - \frac{k}{2} \overline{g(\tau)} \right) \dd \overline{\tau}.
\end{split}
\end{align}
If the limit as $\varepsilon \to 0$ and $y \to \infty$ exists, it yields the regularized inner product $\langle \xi_2 \Delta_2 g, \log(v^{k/2}|f|)\rangle^{\mathrm{reg}}$. Here, the boundary $\partial \mathcal{F}_0^{\varepsilon, y}(N)$ is oriented such that the interior is on the left-hand side. For clarity, referring to the example in \cref{fig:fundamental-domain}, the line segment $t+iy$, which is parallel to the real axis, is oriented from right to left, while the circle around $z_3$ is oriented clockwise.

\section{Proof of \cref{thm:main}}\label{sec:proof}

We evaluate the limit of the right-hand side of \eqref{eq:Stokes} as $\varepsilon \to 0$ and $y \to \infty$. Since the integrand functions are modular forms of weight 2, the integrals over the parts of the boundary of $\mathcal{F}_0^*(N)$ cancel out, leaving only the integrals over the boundaries $C_{\rho, y} \coloneqq \partial B_{\rho, y} \cap \partial \mathcal{F}_0^{\varepsilon, y}(N)$ and $C_{z, \varepsilon} \coloneqq \partial B_{z, \varepsilon} \cap \partial \mathcal{F}_0^{\varepsilon, y}(N)$. We note that similar calculation has been carried out in~\cite[Section 3]{Choi2010}.

\subsection{At singularities on $\bbH$}

Since we are considering the limit as $y \to \infty$, we can assume that $y$ is sufficiently large such that all zeros and poles of $f$ on $\mathcal{F}_0^*(N)$ are contained within $\mathcal{F}_0^{\varepsilon, y}(N)$. Around each zero or pole $z_j$ of $f$, the second integral
\[
	\int_{C_{z_j, \varepsilon}} \left(-\log(v^{k/2}|f(\tau)|) \cdot \overline{\Delta_2 g(\tau)} - \frac{k}{2} \overline{g(\tau)} \right) \dd \overline{\tau}
\]
tends to $0$ as $\varepsilon \to 0$. Indeed, since $g$ is a smooth function that does not have singularities on $\bbH$, it is clear that the limit of the integral of the second term is $0$. For the first term, since the singularities arise from zeros or poles of $f$, it is evaluated by
\[
	\int_{C_{z_j, \varepsilon}} \log |\tau - z_j| \dd \tau = O(\varepsilon \log \varepsilon)
\]
as $\varepsilon \to 0$.

Therefore, only the first integral on the right-hand side of \eqref{eq:Stokes} contributes. Suppose the Laurent expansion
\[
	\frac{\Theta f(\tau)}{f(\tau)} = \sum_{n=-1}^\infty a_n(z_j) (\tau-z_j)^n, \qquad a_{-1}(z_j) = \frac{1}{2\pi i} \mathrm{ord}_{z_j}(f).
\]
If $z_j$ is an interior point of $\mathcal{F}_0^*(N)$, then for a sufficiently small $\varepsilon > 0$, the boundary is contained within the interior of $\mathcal{F}_0^*(N)$, and thus the contribution is given by
\begin{align*}
	\int_{C_{z_j, \varepsilon}} 2\pi \cdot \frac{\Theta f(\tau)}{f(\tau)} \cdot \xi_2 g(\tau) \dd \tau &= - 2\pi \int_0^{2\pi} \sum_{n=-1}^\infty a_n(z_j) (\varepsilon e^{i\theta})^n \cdot \xi_2 g(z_j+\varepsilon e^{i\theta}) i\varepsilon e^{i\theta} \dd \theta\\
	&\xrightarrow{\varepsilon \to 0} -2\pi \mathrm{ord}_{z_j}(f) \xi_2 g(z_j).
\end{align*}
When $z_j$ is on the boundary of $\mathcal{F}_0^*(N)$, we let $[z_j]$ denote the set of all zeros and poles of $f$ in $\bbH \cap \mathcal{F}_0^*(N)$ that are $\Gamma_0(N)$-equivalent to $z_j$. By combining all the contributions on $[z_j]$, we obtain
\begin{align*}
	\lim_{\varepsilon \to 0} \sum_{z \in [z_j]} \int_{C_{z, \varepsilon}} 2\pi \cdot \frac{\Theta f(\tau)}{f(\tau)} \cdot \xi_2 g(\tau) \dd \tau = -2\pi \frac{\mathrm{ord}_{z_j}(f)}{\omega_{z_j}} \xi_2 g(z_j).
\end{align*}

\subsection{At cusps}

Similarly, for each cusp $\rho_j$ of $\mathcal{F}_0^*(N)$, we let $[\rho_j]$ denote the set of all cusps of $\mathcal{F}_0^*(N)$ that are $\Gamma_0(N)$-equivalent to $\rho_j$. Then we have
\begin{align*}
	&\sum_{\rho \in [\rho_j]} \int_{C_{\rho, y}} 2\pi \mathcal{D}_k f(\tau) \cdot \xi_2 g(\tau) \dd \tau + \left(-\log(v^{k/2}|f(\tau)|) \cdot \overline{\Delta_2 g(\tau)} - \frac{k}{2} \overline{g(\tau)} \right) \dd \overline{\tau}\\
	&= - 2\pi \int_{iy}^{\ell_{\rho_j}+iy} \mathcal{D}_k f(\sigma_{\rho_j} \tau) \cdot \xi_2 g(\sigma_{\rho_j} \tau) \dd \sigma_{\rho_j} \tau\\
		&\qquad + \int_{iy}^{\ell_{\rho_j}+iy} \left(\log(\Im(\sigma_{\rho_j} \tau)^{k/2}|f(\sigma_{\rho_j} \tau)|) \cdot \overline{\Delta_2 g(\sigma_{\rho_j} \tau)} + \frac{k}{2} \overline{g(\sigma_{\rho_j} \tau)} \right) \dd \overline{\sigma_{\rho_j} \tau},
\end{align*}
which tends to $C_{1,\rho_j}(f,g) + \overline{C_{2,\rho_j}(f,g)}$ as $y \to \infty$ if the limits exist. 

In conclusion, by summing all these contributions, we obtain \cref{thm:main}.

\section{Applications}\label{sec:Applications}

As applications of \cref{thm:main}, by taking weight 2 functions $g$ arising from the Maass--Poincar\'{e} series, we offer generalizations and alternative proofs for several results. We begin by recalling the Maass--Poincar\'{e} series.

\subsection{Maass--Poincar\'{e} series}

For $k \in 2\Z$ and $n \in \Z$, we define $\psi_{k, n}(\cdot, s): \R_{>0} \to \C$ by
\[
	\psi_{k, n}(v,s) = \begin{cases}
		\Gamma(2s)^{-1} (4\pi |n|v)^{-k/2} M_{\sgn(n)\frac{k}{2}, s-\frac{1}{2}} (4\pi|n|v) &\text{if } n \neq 0,\\
		v^{s-\frac{k}{2}} &\text{if } n =0,
	\end{cases}
\]
where $M_{\mu, \nu}(y)$ is the M-Whitteker function, (see~\cite[Section 6.9]{Erdelyi1953} and \cite[Chapter VII]{MOS1966}). For $\Re(s) > 1$, the weight $k$ Maass--Poincar\'{e} series $P_{N, k, n}(\tau, s)$ is defined by
\[
	P_{N,k,n}(\tau,s) \coloneqq \sum_{\gamma \in \Gamma_0(N)_\infty \backslash \Gamma_0(N)} \psi_{k, n}(v, s) e(nu) |_k \gamma,
\]
where $e(x) \coloneqq e^{2\pi inx}$. In particular, when $n=0$, the series $E_{N,k}(\tau,s) \coloneqq P_{N,k,0}(\tau,s)$ is called the Eisenstein series. Moreover, when $k = 0$ and $-n < 0$, noting that
\[
	M_{0, s - \frac{1}{2}}(4\pi n v) = \frac{2\pi \sqrt{n} \Gamma(2s)}{\Gamma(s)} v^{1/2} I_{s-1/2} (2\pi nv),
\]
where $I_s(v)$ is the $I$-Bessel function, (see~\cite[Section 7.2.4]{MOS1966}), we have the expression
\[
	F_{N,-n}(\tau, s) \coloneqq \Gamma(s) P_{N,0,-n}(\tau, s) = 2\pi \sqrt{n} \sum_{\gamma \in \Gamma_0(N)_\infty \backslash \Gamma_0(N)} v^{1/2} I_{s-1/2}(2\pi nv) e(-nu) |_0 \gamma.
\]
This is known as the Niebur--Poincar\'{e} series. The known properties of the Maass--Poincar\'{e} series are summarized in~\cite{JeonKangKim2013}. In the following applications, we focus on the cases $k=0$ and $k=2$. The properties described below for these special cases can be verified by combining Niebur's original article~\cite{Niebur1973}, the explanation in~\cite{BKLOR2018}, \cite{Iwaniec2002}, and \eqref{eq:MPoincare-xi} below. An important property here is that $P_{N,k,n}(\tau, s)$ admits a meromorphic continuation to $s=1$ and is holomorphic at $s=1$ unless $k=n=0$. In the case $k=n=0$, it has a simple pole at $s=1$ with residue
\[
	\mathrm{Res}_{s=1} E_{N,0}(\tau, s) = \frac{1}{\mathrm{vol}(\Gamma_0(N)\backslash \bbH)}.
\]

For each cusp $\rho$, the Maass--Poincar\'{e} series has the Fourier expansion of the form
\begin{align}\label{eq:Fourier}
	(P_{N,k,n}|_k \sigma_\rho) (\tau,s) = \delta_{\rho, i\infty} \psi_{k,n}(v,s) e(nu) + \varphi_\rho (s) v^{1-s-k/2} + \sum_{m \neq 0} c_{\rho,m}(v,s) e \left(\frac{mu}{\ell_\rho}\right),
\end{align}
where $c_{\rho,m}(v,s)$ decays exponentially as $v \to \infty$ (cf.~\cite[Section 3.4]{Iwaniec2002}). Based on the remark above, $c_{\rho,m}(v,s)$ are holomorphic at $s=1$, and $\varphi_\rho(s)$ has a simple pole at $s=1$ when $k=n=0$, but is holomorphic at $s=1$ in all other cases. Moreover, by a similar argument as in \cite[(3.1)]{Matsusaka2020}, we have
\begin{align}\label{eq:MPoincare-xi}
\begin{split}
	\xi_k P_{N,k,n}(\tau,s) = \begin{cases}
		(4\pi|n|)^{1-k} \left(\overline{s}-\frac{k}{2}\right) P_{N, 2-k, -n}(\tau, \overline{s}) &\text{if } n \neq 0,\\
		\left(\overline{s} - \frac{k}{2}\right) P_{N,2-k,0}(\tau,\overline{s}) &\text{if } n=0.
	\end{cases}
\end{split}
\end{align}

In the study of polyharmonic Maass forms~\cite{LagariasRhoades2016, Matsusaka2020}, it is effective to consider the Taylor (or Laurent) coefficients of these Poincar\'{e} series at $s=1$. We introduce the notation $\mathrm{Coeff}_{(s-1)^m} (P_{N,k,n}(\tau,s)) \coloneqq G_{N,k,n,m}(\tau)$ to represent the $m$-th coefficient of the Laurent expansion
\[
	P_{N,k,n}(\tau,s) = \sum_{m=-1}^\infty G_{N,k,n,m}(\tau) (s-1)^m.
\]
As mentioned above, we have $\mathrm{Coeff}_{(s-1)^{-1}} (P_{N,k,n}(\tau,s)) = 0$ except for the case $k=n=0$.

Now, by using the Maass--Poincar\'{e} series of weight 2, we provide the following five applications. Before proceeding, we state two facts that will be used repeatedly. 

\begin{lemma}\label{lem:log-Fourier}
Let $f$ be a meromorphic modular form of weight $k$ and $\rho$ be a cusp. For any $\tau \in \bbH$ with sufficiently large imaginary part, we have
\[
	\log(v^{k/2} |(f|_k \sigma_\rho)(\tau)|) = \frac{k}{2} \log v -2\pi \frac{\mathrm{ord}_\rho(f)}{\ell_\rho} v + a_{\rho, 0} + \sum_{m=1}^\infty (a_{\rho, m} q^{m/\ell_\rho} + \overline{a_{\rho, m} q^{m/\ell_\rho}})
\]
for some coefficients $a_{\rho,m}$, where $\mathrm{ord}_\rho(f)$ is defined by \eqref{eq:f-rho-Fourier} below. If $\mathrm{ord}_\rho(f) = 0$, the constant $a_{\rho,0}$ is equal to $\log|f(\rho)|$, with $f(\rho) = \lim_{\tau \to i\infty} (f|_k\sigma_\rho) (\tau)$.
\end{lemma}

\begin{proof}
It immediately follows from the Fourier expansion
\begin{align}\label{eq:f-rho-Fourier}
	(f|_k \sigma_\rho)(\tau) = \sum_{n = \mathrm{ord}_\rho(f)}^\infty c_{\rho, n} q^{n/\ell_\rho} \quad (c_{\rho, \mathrm{ord}_\rho(f)} \neq 0)
\end{align}
and the relation 
\[
	\log|1-z| = \frac{1}{2} (\log (1-z) + \log (1-\overline{z})) = -\frac{1}{2} \sum_{n=1}^\infty \frac{z^n + \overline{z}^n}{n}
\]
for $|z| < 1$.
\end{proof}

\begin{lemma}\label{lem:Dk-Fourier}
Let $f$ be a meromorphic modular form of weight $k$ and $\rho$ be a cusp. For any $\tau \in \bbH$ with sufficiently large imaginary part, we have
\[
	(\mathcal{D}_k f)|_2 \sigma_\rho(\tau) = \frac{\mathrm{ord}_\rho(f)}{\ell_\rho} + \sum_{m=1}^\infty a_{\rho, m} q^{m/\ell_\rho} - \frac{k}{4\pi v}
\]
for some coefficients $a_{\rho, m}$.
\end{lemma}

\begin{proof}
From \eqref{eq:log-der}, we have
\[
	(\mathcal{D}_k f)|_2 \sigma_\rho (\tau) = \frac{1}{\pi i} \frac{\partial}{\partial \tau} \log (v^{k/2} |(f|_k\sigma_\rho) (\tau)|).
\]
Thus, \cref{lem:log-Fourier} gives the desired result.
\end{proof}

\subsection{Valence formula}\label{sec:cor1}

First, we let $g(\tau) = E_{N, 2}(\tau, 1)$, which leads to the well-known valence formula. 

\begin{corollary}[{\cite[Theorem 5.6.11]{CohenStromberg2017}}]\label{cor:valence}
For a weight $k$ meromorphic modular form $f$ on $\Gamma_0(N)$, we have
\[
	\sum_{z \in \Gamma_0(N) \backslash \bbH} \frac{\mathrm{ord}_z(f)}{\omega_{z}} + \sum_{\rho \in \mathcal{C}_N} \mathrm{ord}_\rho(f) = \frac{k}{2} \cdot \frac{\mathrm{vol}(\Gamma_0(N) \backslash \bbH)}{2\pi} = \frac{k}{12} [\SL_2(\Z) : \Gamma_0(N)].
\]
\end{corollary}

\begin{proof}
By \eqref{eq:Fourier} and \eqref{eq:MPoincare-xi}, for each cusp $\rho$, we see that
\begin{align*}
	(g|_2 \sigma_\rho) (\tau) &= \delta_{\rho, i\infty} + c_1 v^{-1} + \sum_{m \neq 0} c_{\rho, m}(v,1) e \left(\frac{mu}{\ell_\rho}\right),\\
	\xi_2 g(\tau) &= \lim_{s \to 1} (s-1) E_{N, 0}(\tau, s) = \frac{1}{\mathrm{vol}(\Gamma_0(N) \backslash \bbH)},\\
	\Delta_2 g(\tau) &= \xi_2 \Delta_2 g(\tau) = 0
\end{align*}
for some constant $c_1$. We now compute the $C_{1,\rho}(f,g)$ and $C_{2,\rho}(f,g)$ in \cref{thm:main}. By \cref{lem:Dk-Fourier},
\begin{align*}
	C_{1,\rho} (f,g) &= -2\pi \lim_{y \to \infty} \int_{iy}^{\ell_\rho+iy} (\mathcal{D}_kf)|_2 \sigma_\rho(\tau) \cdot (\xi_2 g)|_0 \sigma_\rho (\tau) \dd \tau\\
		&= - \frac{2\pi}{\mathrm{vol}(\Gamma_0(N) \backslash \bbH)} \lim_{y \to \infty} \int_{iy}^{\ell_\rho + iy} \left(\frac{\mathrm{ord}_\rho(f)}{\ell_\rho} + \sum_{m=1}^\infty a_{\rho,m} q^{m/\ell_\rho} - \frac{k}{4\pi v}\right) \dd \tau\\
		&= -\frac{2\pi}{\mathrm{vol}(\Gamma_0(N)\backslash \bbH)} \lim_{y \to \infty} \ell_\rho \left(\frac{\mathrm{ord}_\rho(f)}{\ell_\rho} - \frac{k}{4\pi y}\right) = -2\pi \cdot \frac{\mathrm{ord}_\rho(f)}{\mathrm{vol}(\Gamma_0(N) \backslash \bbH)}.
\end{align*}
As for $C_{2,\rho}(f,g)$, since $\Delta_2 g(\tau) = 0$, we have
\begin{align*}
	C_{2,\rho} (f,g) &= \frac{k}{2} \lim_{y \to \infty} \int_{iy}^{\ell_\rho+iy} (g|_2\sigma_\rho) (\tau) \dd \tau = \frac{k}{2} \cdot \ell_\rho \delta_{\rho, i\infty}.
\end{align*}
Therefore, since $\xi_2 \Delta_2 g(\tau) = 0$, we obtain that
\[
	0 = - \frac{2\pi}{\mathrm{vol}(\Gamma_0(N) \backslash \bbH)} \left(\sum_{z \in \Gamma_0(N) \backslash \bbH} \frac{\mathrm{ord}_z(f)}{\omega_z} + \sum_{\rho \in \mathcal{C}_N} \mathrm{ord}_\rho(f) \right) + \frac{k}{2},
\]
which implies the desired result.
\end{proof}

\subsection{Rohrlich's formula of level $N$}\label{sec:Rohrlich}

Next, we let $g(\tau) = \mathrm{Coeff}_{(s-1)^1}(E_{N,2}(\tau,s))$, the first Taylor coefficient of $E_{N,2}(\tau,s)$ at $s=1$. This generally defines a function known as a polyharmonic Maass form, (see~\cite{LagariasRhoades2016, ALR2018, Matsusaka2020}). \cref{thm:main} leads to a generalization of Rohrlich's formula for level $N$.

\begin{corollary}\label{cor:Rohrlich}
	For a weight $k$ meromorphic modular form $f$ on $\Gamma_0(N)$ that has neither zeros nor poles at any cusp. Then, we have
	\begin{align*}
		-\frac{1}{\mathrm{vol}(\Gamma_0(N) \backslash \bbH)} \langle 1, \log(v^{k/2} |f|) \rangle^{\mathrm{reg}} = - 2\pi \sum_{z \in \Gamma_0(N) \backslash \bbH} \frac{\mathrm{ord}_z(f)}{\omega_z} \xi_2 g(z) + \frac{k}{2} - \log |f(i\infty)|,
	\end{align*}
	where $\xi_2 g(\tau)$ equals the constant term of the Eisenstein series $E_{N,0}(\tau,s)$ at $s=1$.
\end{corollary}

\begin{remark}
In the special case $\Gamma_0(1) = \SL_2(\Z)$, we use a condition $f(i\infty) = 1$ for normalization. By the Kronecker limit formula,
\begin{align*}
	\xi_2 g(\tau) &= \mathrm{Coeff}_{(s-1)^0} (E_{1,0}(\tau,s)) = -\frac{1}{2\pi} \log(v^6|\Delta(\tau)|) + C,
\end{align*}
where $C = (6-72\zeta'(-1) - 6\log(4\pi))/\pi$, and $\zeta(s)$ is the Riemann zeta function. Therefore, by \cref{cor:Rohrlich}, we have
\[
	- \frac{3}{\pi} \langle 1, \log(v^{k/2} |f|) \rangle^{\mathrm{reg}} = \sum_{z \in \SL_2(\Z) \backslash \bbH} \frac{\mathrm{ord}_z(f)}{\omega_z} (\log(y^6 |\Delta(z)|) - 2\pi C) + \frac{k}{2},
\]
where we put $z = x + iy$. The valence formula (\cref{cor:valence}) leads to the Rohrlich's theorem
\begin{align}\label{eq:or-Rohrlich}
	\langle 1, \log(v^{k/2} |f|) \rangle^\mathrm{reg} = -\frac{\pi}{3} \left(\sum_{z \in \SL_2(\Z) \backslash \bbH} \frac{\mathrm{ord}_z(f)}{\omega_z} \log(y^6 |\Delta(z)|) - \frac{\pi k}{6}C + \frac{k}{2} \right),
\end{align}
originally obtained by Rohrlich in~\cite{Rohrlich1984} as a modular form analogue of the classical Jensen formula in complex analysis. In Rohrlich's article, the constant term is given by
\[
	12 \left(\frac{1}{2} + \log 2 - \gamma + \frac{\zeta'(2)}{\zeta(2)}\right) \sum_{z \in \SL_2(\Z) \backslash \bbH} \frac{\mathrm{ord}_z(f)}{\omega_z},
\]
but this is equal to $-\pi kC/6 + k/2$ as stated above.
\end{remark}

\begin{proof}[Proof of \cref{cor:Rohrlich}]
By \eqref{eq:Fourier} and \eqref{eq:MPoincare-xi}, for each cusp $\rho$, we see that
\begin{align*}
	(g|_2 \sigma_\rho) (\tau) &= \delta_{\rho, i\infty} \log v + c_1 \frac{\log v}{v} + c_2 \frac{1}{v} + \sum_{m \neq 0} a_{\rho, m}(v) e \left(\frac{mu}{\ell_\rho}\right),\\
	(\xi_2 g|_0 \sigma_\rho)(\tau) &= \mathrm{Coeff}_{(s-1)^0} ((E_{N,0}|_0 \sigma_\rho)(\tau,s)) = \delta_{\rho, i\infty} v + c_3 + c_4 \log v + \sum_{m \neq 0} b_{\rho,m}(v) e \left(\frac{mu}{\ell_\rho}\right),\\
	(\Delta_2 g|_2 \sigma_\rho)(\tau) &= -\mathrm{Coeff}_{(s-1)^0} (s(E_{N,2}|_2 \sigma_\rho)(\tau,s)) = -\delta_{\rho, i\infty} + c_5 \frac{1}{v} + \sum_{m \neq 0} c_{\rho,m}(v) e \left(\frac{mu}{\ell_\rho}\right),\\
	\xi_2\Delta_2 g(\tau) &= - \mathrm{Res}_{s=1} E_{N,0}(\tau,s) = - \frac{1}{\mathrm{vol}(\Gamma_0(N) \backslash \bbH)},
\end{align*}
for some constants $c_1, \dots, c_5$ and coefficients $a_{\rho,m}(v), b_{\rho,m}(v), c_{\rho,m}(v)$ that decay exponentially as $v \to \infty$. As an example, we remark on the constants $c_1$ and $c_2$, which are specifically given by
\[
	\frac{\partial}{\partial s} \varphi_\rho(s) v^{-s} \bigg|_{s=1} = -\varphi_\rho(1) \frac{\log v}{v} + \varphi'_\rho(1) \frac{1}{v}.
\]
By our assumption on $f$, for any cusp $\rho$, we have $\mathrm{ord}_\rho(f) = 0$. Therefore, by \cref{lem:Dk-Fourier}, we obtain
\begin{align*}
	&C_{1,\rho} (f,g)\\
	&= -2\pi \lim_{y \to \infty} \int_{iy}^{\ell_\rho+iy} \left(\sum_{m=1}^\infty a'_{\rho, m} q^{m/\ell_\rho} - \frac{k}{4\pi v}\right) \cdot \left(\delta_{\rho, i\infty} v + c_3 + c_4 \log v + \sum_{m \neq 0} b_{\rho,m}(v) e \left(\frac{mu}{\ell_\rho}\right)\right) \dd \tau\\
	&= -2\pi \ell_\rho \lim_{y \to \infty} \left(-\frac{k}{4\pi y}\left(\delta_{\rho, i\infty} y + c_3 + c_4 \log y\right) + \sum_{m=1}^\infty a'_{\rho, m} e^{-2\pi y m/\ell_\rho} b_{\rho, -m}(y) \right)\\
	&= \frac{k}{2} \cdot \ell_\rho \delta_{\rho, i\infty}.
\end{align*}
On the other hand, by \cref{lem:log-Fourier}, the contributions from $a_{\rho,m}(v)$ and $c_{\rho,m}(v)$ can similarly be ignored, and we have
\begin{align*}
	C_{2,\rho} (f,g) &= \lim_{y \to \infty} \int_{iy}^{\ell_\rho+iy} \left(\log(v^{k/2}|(f|_k \sigma_\rho)(\tau)|) \cdot (\Delta_2 g)|_2\sigma_\rho (\tau) + \frac{k}{2} (g|_2\sigma_\rho) (\tau) \right) \dd \tau\\
		&= \ell_\rho \lim_{y \to \infty} \left[\left(\frac{k}{2} \log y + \log |f(\rho)|\right) \cdot \left(-\delta_{\rho,i\infty} + c_5 \frac{1}{y}\right) + \frac{k}{2} \left(\delta_{\rho, i\infty} \log y + c_1 \frac{\log y}{y} + c_2 \frac{1}{y}\right) \right]\\
		&= -\ell_\rho \delta_{\rho, i\infty} \log|f(\rho)|.
\end{align*}
Here, a minor yet important point is that the growth of $\log y$ cancels out. Therefore, we obtain
\begin{align*}
	-\frac{1}{\mathrm{vol}(\Gamma_0(N) \backslash \bbH)} \langle 1, &\log(v^{k/2} |f|) \rangle^\mathrm{reg} = -2\pi \sum_{z \in \Gamma_0(N) \backslash \bbH} \frac{\mathrm{ord}_z(f)}{\omega_z} \xi_2 g(z) + \frac{k}{2} - \log|f(i\infty)|,
\end{align*}
which concludes the proof.
\end{proof}

\subsection{Explicit version of Bringmann--Kane--L\"{o}brich--Ono--Rolen's theorem}\label{sec:cor3}

Now, we examine $g(\tau) = \mathrm{Coeff}_{(s-1)^1} (P_{N,2,n}(\tau,s))$ with $n > 0$ from the Maass--Poincar\'{e} series, focusing on the first Taylor coefficient at $s=1$, as the constant term yields only the trivial equation $0=0$. 

\begin{corollary}\label{cor:BKLOR}
	For a weight $k$ meromorphic modular form $f$ on $\Gamma_0(N)$ and positive integer $n$, we have
	\[
		\sum_{z \in \Gamma_0(N) \backslash \bbH} \frac{\mathrm{ord}_z(f)}{\omega_z} j_{N,n}(z) + \sum_{\rho \in \mathcal{C}_N} \mathrm{ord}_\rho (f) j_{N,n}(\rho) = - \mathrm{Coeff}_{q^n} \left(\frac{\Theta f(\tau)}{f(\tau)} \right),
	\]
	where $j_{N,n}(\tau) \coloneqq P_{N,0,-n}(\tau,1)$ and $j_{N,n}(\rho)$ is the constant term of the Fourier expansion of $(j_{N,n}|_0 \sigma_\rho)(\tau)$.
\end{corollary}

\begin{remark}
Bringmann--Kane--L\"{o}brich--Ono--Rolen~\cite[Theorem 1.3]{BKLOR2018} showed that there exists a cusp form $h \in S_2(\Gamma_0(N))$ such that
\begin{align*}
	\sum_{z \in \Gamma_0(N) \backslash \bbH} \frac{\mathrm{ord}_z(f)}{\omega_z} j_{N,n}(z) + \sum_{\rho \in \mathcal{C}_N} \mathrm{ord}_\rho (f) j_{N,n}(\rho) = - \mathrm{Coeff}_{q^n} \left(\frac{\Theta f(\tau)}{f(\tau)} \right) + a_h(n),
\end{align*}
where $a_h(n)$ is the $n$-th Fourier coefficient of $h(\tau)$. Our \cref{cor:BKLOR} provides an explicit version of their Theorem 1.3, namely, that the cusp form $h$ is always $0$.
\end{remark}

\begin{example}
We provide a specific example for the case $N=11$, where the space of cusp forms is non-trivial. Let $M_0^{!, \infty}(\Gamma_0(11))$ denote the subspace of $M_0^!(\Gamma_0(11))$ consisting of weakly holomorphic modular forms whose poles are supported only at $i\infty$ and are holomorphic at other cusps. As shown in \cite{Yang2006} and \cite[Section 3]{Kang2021}, the space has a natural basis given by $\{1\} \cup \{f_{11,m}(\tau) : m \ge 2\}$, where $f_{11,m}(\tau)$ is given as
\begin{align*}
	f_{11,m}(\tau) &= (j_{11,m}(\tau) - c_{11,m}(0)) + a_{11}(m,-1) (j_{11,1}(\tau) - c_{11,1}(0))\\
		&= q^{-m} + a_{11}(m,-1) q^{-1} + O(q).
\end{align*}
The constants $c_{11,m}(0)$ and $a_{11}(m,-1)$ are defined as follows. First, $c_{11,m}(0)$ is the constant term of the Fourier expansion of $j_{11,m}(\tau)$, calculated as
\[
	c_{11,m}(0) = -\frac{1}{5} \left(\sigma_1(m) - 11^2 \sigma_1 \left(\frac{m}{11}\right)\right)
\]
in \cite[(2.6)]{Lobrich2018}. As for $a_{11}(m,-1)$, by combining \cite[Theorem 1.12]{JKK2023Ramanujan} and the fact that $S_2(\Gamma_0(11))$ is generated by
\[
	h_{11,-1}(\tau) \coloneqq \eta(\tau)^2 \eta(11\tau)^2 = q \prod_{n=1}^\infty (1-q^n)^2 (1-q^{11n})^2 = q + \sum_{m=2}^\infty b_{11}(-1,m) q^m,
\]
we have $a_{11}(m,-1) = - b_{11}(-1,m)$. Therefore, \cref{cor:BKLOR} and \cref{cor:valence} yield
\begin{align}\label{eq:f11-exp}
\begin{split}
	&\sum_{z \in \Gamma_0(11) \backslash \bbH} \frac{\mathrm{ord}_z(f)}{\omega_z} f_{11,m}(z) + \sum_{\rho \in \mathcal{C}_{11}} \mathrm{ord}_\rho (f) f_{11,m}(\rho)\\
	&= - \mathrm{Coeff}_{q^m} \left(\frac{\Theta f(\tau)}{f(\tau)} \right) + b_{11}(-1,m) \mathrm{Coeff}_{q} \left(\frac{\Theta f(\tau)}{f(\tau)}\right) + \frac{k}{5} \left(\sigma_1(m) - 11^2 \sigma_1 \left(\frac{m}{11}\right) - b_{11}(-1,m) \right).
\end{split}
\end{align}
For instance, let $m=2, 3, 4$ and take a weight 2 modular form
\[
	f(\tau) = -\frac{1}{10} \left(E_2(\tau) - 11 E_2(11\tau) + 24 h_{11,-1}(\tau)\right),
\]
where $E_2(\tau) = 1 - 24 \sum_{n=1}^\infty \sigma_1(n) q^n$ is the weight 2 Eisenstein series. As shown in \cite[Example 6.1]{JeonKangKim2022}, this form satisfies $\mathrm{ord}_\rho(f) = 0$ for $\rho \in \{0, i\infty\}$ and has two zeros $z_+, z_-$ on $\Gamma_0(11) \backslash \bbH$, which are either two simple distinct zeros or a single double zero. Since
\[
	h_{11,-1}(\tau) = q - 2q^2 - q^3 + 2q^4 + \cdots
\]
and
\[
	\frac{\Theta f(\tau)}{f(\tau)} = 24q^2 + 36q^3 - 240q^4 + \cdots,
\]
we find from \eqref{eq:f11-exp} that $f_{11,2}(z_+) + f_{11,2}(z_-) = -22, f_{11,3}(z_+) + f_{11,3}(z_-) = -34, f_{11,4}(z_+) + f_{11,4}(z_-) = 242$.
On the other hand, as $f_{11,2}(\tau) = q^{-2} + q^{-1} + 5q + 8q^2 + \cdots$, (see~\cite{Kang2021, Yang2006}), we can confirm that 
\begin{align}\label{eq:f11-rel}
	f_{11,4}(\tau) = f_{11,2}(\tau)^2 - 4f_{11,3}(\tau) - 4f_{11,2}(\tau) - 36.
\end{align}
Therefore, we have $f_{11,2}(z_+) f_{11,2}(z_-) = 197$, which implies that $f_{11,2}(z_+)$ and $f_{11,2}(z_-)$ are roots of the quadratic polynomial $X^2 + 22X + 197$. In particular, they are distinct.
\end{example}

\begin{remark}
Approximate calculations using the Fourier expansions of $f(\tau)$ and $f_{11,2}(\tau)$, along with numerical experiments in Mathematica, show that that $f(\tau)$ has zeros at 
\begin{align*}
	z_\pm &= \pm 0.22727\dots + i 0.19813\dots \thickapprox \frac{\pm 5 + \sqrt{-19}}{22},
\end{align*}
and $f_{11,2}(z_\pm) = -11. \mp i8.7178\dots \thickapprox -11 \mp 2\sqrt{-19}$. This is consistent with the above result. The same result was calculated using a different method in a previous work by the first three authors~\cite[Example 6.1]{JeonKangKim2022}, but several typographical errors need to be corrected there. First, the calculation of the Fourier coefficients of $E_f(\tau)$ is incorrect, and second, the constant term $-36$ in equation~\eqref{eq:f11-rel} is missing. By correcting these, one can obtain the same result as above.
\end{remark}

\begin{proof}[Proof of \cref{cor:BKLOR}]
The approach is exactly the same as those in the previous corollaries. For each cusp $\rho$, we have
\begin{align*}
	(g|_2 \sigma_\rho)(\tau) &= \delta_{\rho, i\infty} \mathrm{Coeff}_{(s-1)^1}(\psi_{2,n}(v,s)) e(nu) + c_1 \frac{\log v}{v} + c_2 \frac{1}{v} + \sum_{m \neq 0} a_{\rho,m}(v) e\left(\frac{mu}{\ell_\rho}\right),\\
	(\xi_2 g|_0 \sigma_\rho)(\tau) &= (4\pi n)^{-1} P_{N, 0, -n}(\sigma_\rho \tau,1) = \frac{\delta_{\rho, i\infty} \psi_{0,-n}(v,1) e(-nu)}{4\pi n} + \frac{C_{\rho, n}}{4\pi n} + \sum_{m \neq 0} b_{\rho,m}(v) e \left(\frac{mu}{\ell_\rho}\right),\\
	(\Delta_2 g|_2 \sigma_\rho)(\tau) &= - P_{N,2,n}|_2 \sigma_\rho (\tau,1) = -\delta_{\rho, i\infty} \psi_{2,n}(v,1) e(nu) + c_3 \frac{1}{v} + \sum_{m \neq 0} c_{\rho,m}(v) e \left(\frac{mu}{\ell_\rho}\right),\\
	\xi_2 \Delta_2 g(\tau) &= - (4\pi n)^{-1} \mathrm{Res}_{s=1}(s P_{N,0,n}(\tau,s)) = 0
\end{align*}
for some constants $c_1, c_2, c_3, C_{\rho, n}$ and coefficients $a_{\rho, m}(v), b_{\rho,m}(v), c_{\rho, m}(v)$ that decay exponentially as $v \to \infty$. Regarding the above expressions, we have a few additional notes. First, since $\xi_2 \Delta_2 g(\tau) = 0$, the function $\Delta_2 g(\tau)$ must be holomorphic in $\tau$. Since 
\begin{align}\label{eq:psi2}
	\psi_{2,n}(v,1) e(nu) = (4\pi nv)^{-1} M_{1, 1/2}(4\pi nv) e(nu) = q^n,
\end{align}
(see~\cite[7.2.4]{MOS1966} or \cite[(2.21)]{JeonKangKim2013}), the constant $c_3$ must be 0 and $\Delta_2 g(\tau)$ defines a holomorphic cusp form. Moreover, it is also known (\cite[(2.19)]{JeonKangKim2013}) that
\begin{align*}
	\psi_{0,-n}(v,1) e(-nu) &= (e^{2\pi nv} - e^{-2\pi nv})e^{-2\pi inu} = q^{-n} - \overline{q}^n.
\end{align*}
Then, by \cref{lem:Dk-Fourier}, we have
\begin{align*}
	C_{1,\rho} (f,g) &= -2\pi \lim_{y \to \infty} \int_{iy}^{\ell_\rho+iy} \left(\frac{\mathrm{ord}_\rho(f)}{\ell_\rho} + \sum_{m=1}^\infty a'_{\rho,m} q^{m/\ell_\rho} - \frac{k}{4\pi v}\right) \cdot \left(\frac{\delta_{\rho, i\infty} \psi_{0,-n}(v,1) e(-nu)}{4\pi n} + \frac{C_{\rho, n}}{4\pi n} \right) \dd \tau\\
	&= -2\pi \ell_\rho \lim_{y \to \infty} \left( \left(\frac{\mathrm{ord}_\rho(f)}{\ell_\rho} - \frac{k}{4\pi y}\right) \cdot \frac{C_{\rho,n}}{4\pi n} + a'_{\rho,n} e^{- \frac{2\pi n y}{\ell_\rho}} \delta_{\rho, i\infty} \cdot \frac{e^{2\pi ny} - e^{-2\pi ny}}{4\pi n} \right)\\
	&= - \frac{1}{2n} \left(\mathrm{ord}_\rho(f) C_{\rho,n} + \delta_{\rho, i\infty} a'_{i\infty,n} \right).
\end{align*}
Next, we note that $\Delta_2 g(\tau)$ is a holomorphic cusp form. Then by \cref{lem:log-Fourier}, we see that
\begin{align*}
	C_{2,\rho} (f,g) &= \lim_{y \to \infty} \int_{iy}^{\ell_\rho+iy} \left(\log(v^{k/2}|(f|_k \sigma_\rho)(\tau)|) \cdot (\Delta_2 g)|_2\sigma_\rho (\tau) + \frac{k}{2} (g|_2\sigma_\rho) (\tau) \right) \dd \tau = 0.
\end{align*}
Given that $\xi_2 \Delta_2 g(\tau) = 0$, we arrive at
\[
	0 = -2\pi \sum_{z \in \Gamma_0(N) \backslash \bbH} \frac{\mathrm{ord}_z(f)}{\omega_z} \xi_2 g(z) - \frac{1}{2n} \sum_{\rho \in \mathcal{C}_N} \mathrm{ord}_\rho(f) C_{\rho,n} - \frac{a'_{i\infty,n}}{2n}.
\]
Since $\xi_2 g(\tau) = (4\pi n)^{-1} j_{N,n}(\tau)$, $C_{\rho, n} = j_{N,n}(\rho)$, and $a'_{i\infty,n}$ is the $n$-th Fourier coefficient of $\frac{\Theta f(\tau)}{f(\tau)}$ by definition, we obtain the desired result.
\end{proof}

\subsection{Bringmann--Kane's formula of level $N$}\label{sec:cor4}

Similarly to \cref{sec:Rohrlich}, we consider the higher-order Taylor coefficients of the Maass--Poincar\'{e} series $g(\tau) = \mathrm{Coeff}_{(s-1)^2} (P_{N,2,n}(\tau,s))$ with $n > 0$. The consequence in this case extends the generalized Rohrlich formula given by Bringmann--Kane~\cite{BringmannKane2020} at level 1 to a general level $N$. Their proof relies on the fact that $\SL_2(\Z)$ has only one cusp, and in the case of Fuchsian groups with one cusp, a generalization is already known by Cogdell--Jorgenson--Smajlovi\'{c}~\cite{CJS2023}. Our following proof works for the general $\Gamma_0(N)$ regardless of the number of cusps.

\begin{corollary}\label{cor:BringmannKane}
	For a positive integer $n$ and a weight $k$ meromorphic modular form $f$ on $\Gamma_0(N)$ that has neither zeros nor poles at any cusp, we have
	\[
		\langle j_{N,n}, \log(v^{k/2} |f|)\rangle^{\mathrm{reg}} = -2\pi \sum_{z \in \Gamma_0(N) \backslash \bbH} \frac{\mathrm{ord}_z(f)}{\omega_z} \mathbb{J}_{N,n}(z),
	\]
	where $j_{N,n}(\tau) \coloneqq P_{N,0,-n}(\tau, 1)$ as before, and $\mathbb{J}_{N,n}(\tau) \coloneqq - \mathrm{Coeff}_{(s-1)^1} (P_{N,0,-n}(\tau,s)) + \gamma j_{N,n}(\tau)$ with Euler's constant $\gamma$.
\end{corollary}

\begin{remark}
In the case of $N=1$, it corresponds to Theorem 1.3 of Bringmann--Kane~\cite{BringmannKane2020}. By expressing their notation (with superscript $\mathrm{BK}$) in our notation, we obtain the relationships
\[
	j_n^\mathrm{BK}(\tau) = j_{1,n}(\tau) - 24 \sigma_1(n), \qquad \mathbbm{j}_n^\mathrm{BK}(\tau) = \mathbb{J}_{1,n}(\tau) - 4\sigma_1(n) \log(v^6 |\Delta(\tau)|) + C_n
\]
for some constant $C_n$. We can verify the equivalence by referring to their definitions and Rohrlich's formula in \eqref{eq:or-Rohrlich}. It is a modest yet important achievement that we have summarized their various normalization conditions and the constant $c_n$ (which appears in their Theorem 1.3 and differs from the constant $C_n$ above) into a simpler expression, despite its somewhat complex original definition.
\end{remark}

\begin{proof}[Proof of \cref{cor:BringmannKane}]
For each cusp $\rho$, we have
\begin{align*}
	(g|_2\sigma_\rho)(\tau) &= \delta_{\rho, i\infty} \mathrm{Coeff}_{(s-1)^2} (\psi_{2,n}(v,s)) e(nu) + \frac{c_1 (\log v)^2 + c_2 \log v + c_3}{v} + \sum_{m \neq 0} a_{\rho,m}(v) e\left(\frac{mu}{\ell_\rho}\right),\\
	(\xi_2g|_0\sigma_\rho)(\tau) &= (4\pi n)^{-1} \mathrm{Coeff}_{(s-1)^1} (P_{N,0,-n}(\sigma_\rho \tau,s))\\
		&= \frac{\delta_{\rho, i\infty} \mathrm{Coeff}_{(s-1)^1} (\psi_{0,-n}(v,s)) e(-nu)}{4\pi n} + c_4 + c_5 \log v + \sum_{m \neq 0} b_{\rho,m}(v) e\left(\frac{mu}{\ell_\rho}\right),\\
	(\Delta_2 g|_2\sigma_\rho)(\tau) &= \mathrm{Coeff}_{(s-1)^1} (-s P_{N,2,n}|_2 \sigma_\rho (\tau,s)) = -P_{N,2,n}|_2 \sigma_\rho(\tau,1) - \mathrm{Coeff}_{(s-1)^1} (P_{N,2,n}|_2 \sigma_\rho (\tau,s))\\
		&= - \delta_{\rho, i\infty} (\mathrm{Coeff}_{(s-1)^1} (\psi_{2,n}(v,s)) + \psi_{2,n}(v,1)) e(nu) + \frac{c_6 \log v + c_7}{v} + \sum_{m \neq 0} c_{\rho,m}(v) e \left(\frac{mu}{\ell_\rho}\right),\\
	(\xi_2 \Delta_2 g|_0 \sigma_\rho)(\tau) &= - (4\pi n)^{-1} P_{N,0,-n}(\sigma_\rho \tau,1)\\
		&= -\frac{\delta_{\rho, i\infty} \psi_{0,-n}(v,1) e(-nu)}{4\pi n} - \frac{C_{\rho, n}}{4\pi n} + \sum_{m \neq 0} d_{\rho,m}(v) e \left(\frac{mu}{\ell_\rho}\right)
\end{align*}
for some constans $c_1, \dots, c_7$, $C_{\rho,n}$, and coefficients $a_{\rho,m}(v), \dots, d_{\rho, m}(v)$ that decay exponentially as $v \to \infty$. Since the constant term of \cref{lem:Dk-Fourier} is 0 due to our assumption on $f$, it follows that
\begin{align*}
	&C_{1,\rho} (f,g)\\
	&= -2\pi \lim_{y \to \infty} \int_{iy}^{\ell_\rho+iy} \left(\sum_{m=1}^\infty a'_{\rho,m} q^{m/\ell_\rho} - \frac{k}{4\pi v} \right) \cdot \left(\frac{\delta_{\rho, i\infty} \mathrm{Coeff}_{(s-1)^1} (\psi_{0,-n}(v,s)) e(-nu)}{4\pi n} + c_4 + c_5 \log v\right) \dd \tau\\
	&= -2\pi \ell_\rho \lim_{y \to \infty} \left( \frac{-k}{4\pi y} (c_4 + c_5 \log y) + \delta_{\rho, i\infty} \frac{a'_{i\infty, n} e^{-2\pi ny} \mathrm{Coeff}_{(s-1)^1}(\psi_{0,-n}(y,s))}{4\pi n} \right)\\
	&= - \frac{\delta_{\rho, i\infty} a'_{i\infty,n}}{2n} \lim_{y \to \infty} e^{-2\pi ny} \mathrm{Coeff}_{(s-1)^1}(\psi_{0,-n}(y,s)).
\end{align*}
Here, we have
\begin{align*}
	\mathrm{Coeff}_{(s-1)^1} \left(\psi_{0,-n}(y,s)\right) &= \mathrm{Coeff}_{(s-1)^1} \left(2\pi \sqrt{ny} \cdot \frac{I_{s-1/2}(2\pi ny)}{\Gamma(s)} \right)\\
		&= 2\gamma \sinh(2\pi ny) + 2\pi \sqrt{ny} \frac{\partial}{\partial s} I_{s-1/2} (2\pi ny) \bigg|_{s=1}.
\end{align*}
By \cite[6.12.1, 6.12.2, and 10.38.6]{NIST}, since
\[
	2\pi \sqrt{ny} \frac{\partial}{\partial s} I_{s-1/2} (2\pi ny) \bigg|_{s=1} = -\left(E_1(4\pi ny) e^{2\pi ny} + \mathrm{Ei}(4\pi ny) e^{-2\pi ny}\right) \sim - \frac{e^{2\pi ny}}{4\pi ny}
\]
as $y \to \infty$, we have
\begin{align}\label{eq:cor4-1}
	C_{1,\rho}(f,g) = - \frac{\delta_{\rho, i\infty} \gamma}{2n} \mathrm{Coeff}_{q^n} \left(\frac{\Theta f(\tau)}{f(\tau)}\right).
\end{align}
In a similar manner, using \cref{lem:log-Fourier}, we find that
\begin{align*}
	&C_{2,\rho} (f,g)\\
	&= \lim_{y \to \infty} \int_{iy}^{\ell_\rho+iy} \left(\left(\frac{k}{2} \log v + \log |f(\rho)| + \sum_{m \neq 0} b'_{\rho,m}e^{-2\pi \frac{|m|}{\ell_\rho}v} e \left(\frac{mu}{\ell_\rho}\right)\right) \cdot (\Delta_2 g)|_2\sigma_\rho (\tau) + \frac{k}{2} (g|_2\sigma_\rho) (\tau) \right) \dd \tau\\
	&= \ell_\rho \lim_{y \to \infty} \bigg[\left(\frac{k}{2} \log y + \log |f(\rho)|\right) \frac{c_6 \log y + c_7}{y} - \delta_{\rho, i\infty} b'_{i\infty, -n} e^{-2\pi ny} (\mathrm{Coeff}_{(s-1)^1} (\psi_{2,n}(y,s)) + \psi_{2,n}(y,1))\\
	&\qquad + \frac{k}{2} \frac{c_1 (\log y)^2 + c_2 \log y + c_3}{y} \bigg]\\
	&= -\delta_{\rho, i\infty} b'_{i\infty, -n} \lim_{y \to \infty} e^{-2\pi ny} (\mathrm{Coeff}_{(s-1)^1} (\psi_{2,n}(y,s)) + \psi_{2,n}(y,1)).
\end{align*}
According to \eqref{eq:psi2}, we obtain
\begin{align*}
	\mathrm{Coeff}_{(s-1)^1} (\psi_{2,n}(y,s)) &= 2(\gamma-1) e^{-2\pi ny} + (4\pi ny)^{-1} \frac{\partial}{\partial s} M_{1, s-\frac{1}{2}}(4\pi ny) \bigg|_{s=1}.
\end{align*}
Using the same notation as in~\cite{ALR2018} and following \cite[9.233]{GR2015}, we have
\begin{align*}
	M_{1, s-\frac{1}{2}}(4\pi ny) = -\frac{\Gamma(2s)}{\Gamma(s-1)} \mathcal{M}_{1, s-1/2}^+(4\pi ny) - \frac{\Gamma(2s)}{\Gamma(s+1)} e^{-\pi is} W_{1,s-1/2}(4\pi ny).
\end{align*}
By \cite[Proposition A.1]{ALR2018}, we have
\[
	(4\pi ny)^{-1} \frac{\partial}{\partial s} M_{1, s-\frac{1}{2}}(4\pi ny) \bigg|_{s=1} \sim - (4\pi ny)^{-1} \mathcal{M}_{1,1/2}^+(4\pi ny) \sim \frac{e^{2\pi ny}}{(4\pi ny)^2} 
\]
as $y \to \infty$, and hence we have
\begin{align}\label{eq:cor4-2}
	C_{2,\rho}(f,g) = 0.
\end{align}
Therefore, by \cref{thm:main} and equations \eqref{eq:cor4-1} and \eqref{eq:cor4-2}, we find that
\begin{align*}
	&-\frac{1}{4\pi n} \langle j_{N,n}, \log(v^{k/2} |f|)\rangle^{\mathrm{reg}}\\
	&= - \frac{1}{2n} \sum_{z \in \Gamma_0(N) \backslash \bbH} \frac{\mathrm{ord}_z(f)}{\omega_z} \mathrm{Coeff}_{(s-1)^1}(P_{N,0,-n}(z,s)) - \frac{\gamma}{2n} \mathrm{Coeff}_{q^n} \left(\frac{\Theta f(\tau)}{f(\tau)}\right),
\end{align*}
where we applied $\xi_2 g(\tau) = (4\pi n)^{-1} \mathrm{Coeff}_{(s-1)^1} (P_{N,0,-n}(\tau,s))$ and $\xi_2 \Delta_2 g(\tau) = -(4\pi n)^{-1} j_{N,n}(\tau)$. Finally, \cref{cor:BKLOR} yields the desired result.
\end{proof}

\subsection{Twisted traces of CM values of the Eisenstein series} \label{sec:cor5}

Finally, in the case of level 1, we consider the Eisenstein series $g(\tau) = E_{1,2}(\tau,s)$ with complex parameter $s$ and demonstrate that, by applying the generalized Borcherds product, we can provide a more shorter alternative proof of the result by Duke--Imamo\={g}lu--T\'{o}th concerning the twisted trace of CM values. The key idea is to obtain an expression using the regularized Petersson inner product, which allows for an effective use of the self-adjointness of Hecke operators. First, the consequence of \cref{thm:main} is as follows.

\begin{corollary}\label{cor:Eisen-case}
	For a weight $0$ meromorphic modular form on $\SL_2(\Z)$ with $f(i\infty) = \lim_{\tau \to i\infty} f(\tau) = 1$, we have
	\[
		s(1-s) \langle E_{1,0}(\cdot, s), \log|f| \rangle^{\mathrm{reg}} = -2\pi \sum_{z \in \SL_2(\Z) \backslash \bbH} \frac{\mathrm{ord}_z(f)}{\omega_z} E_{1,0}(z,s).
	\]
\end{corollary}

\begin{proof}
In this case, the only cusp is $i\infty$. By our assumption and \cref{lem:Dk-Fourier}, we have
\[
	C_{1,i\infty}(f,g) = -2\pi \lim_{y \to \infty} \int_{iy}^{1+iy} \left(\sum_{m=1}^\infty a_m q^m\right) \cdot (\overline{s}-1) E_{1,0}(\tau, \overline{s}) \dd \tau = 0.
\]
Similarly, by \cref{lem:log-Fourier}, 
\[
	C_{2,i\infty}(f,g) = \lim_{y \to \infty} \int_{iy}^{1+iy} \left(\sum_{m \neq 0} a_m(v) e(mu)\right) \cdot s(1-s) E_{1,2}(\tau,s) \dd \tau = 0.
\]
Therefore, we obtain
\[
	\langle \xi_2 \Delta_2 E_{1,2}(\cdot, s), \log |f| \rangle^{\mathrm{reg}} = -2\pi \sum_{z \in \SL_2(\Z) \backslash \bbH} \frac{\mathrm{ord}_z(f)}{\omega_z} \xi_2 E_{1,2}(z, s),
\]
which implies the desired result by \eqref{eq:MPoincare-xi}.
\end{proof}

We now briefly review the result on twisted traces of CM values. For a negative discriminant $d < 0$ such that $d \equiv 0, 1 \pmod{4}$, we let $\mathcal{Q}_d$ denote the set of positive definite integral binary quadratic forms $Q(X,Y) = AX^2 + BXY + CY^2$. As is classically known, the group $\SL_2(\Z)$ acts on this set by
\[
	\left(Q \circ \pmat{a & b \\ c & d}\right) (X,Y) = Q(aX+bY, cX+dY),
\]
and the number of classes under this action is finite. For each $Q \in \mathcal{Q}_d$, we define $\alpha_Q \in \bbH$ as the unique root in $\bbH$ of $Q(\tau,1) = 0$, and set $\omega_Q := \# \mathrm{PSL}_2(\Z)_Q = \omega_{\alpha_Q}$. 

For a pair of a fundamental discriminant $D > 1$ and a discriminant $d < 0$ (not necessarily fundamental), we define a genus character $\chi_D : \mathcal{Q}_{dD}/\SL_2(\Z) \to \{-1, 0, 1\}$ by
\[
	\chi_D(Q) = \begin{cases}
		\left(\frac{D}{r}\right) &\text{if } (A,B,C,D) = 1, \text{ where $Q$ represents $r$ and $(r,D) = 1$},\\
		0 &\text{otherwise},
	\end{cases}
\]
where $\left(\frac{\cdot}{\cdot}\right)$ is the Kronecker symbol. For a weight 0 modular form $F$, the twisted trace of CM values is defined by
\[
	\mathrm{Tr}_{d,D}(F) \coloneqq \sum_{Q \in \mathcal{Q}_{dD}/\SL_2(\Z)} \frac{\chi_D(Q)}{\omega_Q} F(\alpha_Q).
\]
In particular, when $F(\tau)$ is the Eisenstein series $E_{1,0}(\tau,s)$, Duke--Imamo\={g}lu--T\'{o}th~\cite{DukeImamogluToth2011} obtained the following decomposition formula. It follows from the combination of (2.25) and (5.2) in their article.

\begin{theorem}\label{thm:DIT}
For fundamental discriminants $D > 1, D' < 0$ and $d = D' m^2$, we have
\[
	\mathrm{Tr}_{d,D}(E_{1,0}(\cdot,s)) = \frac{|DD'|^{s/2}}{2^s} \frac{L_D(s) L_{D'}(s)}{\zeta(2s)} \sum_{n \mid m} \mu \left(\frac{m}{n}\right) \left(\frac{D'}{m/n}\right) n^{1-s} \sigma_{2s-1}(n),
\]
where $L_D(s) = \sum_{n=1}^\infty \left(\frac{D}{n}\right) n^{-s}$ is the Dirichlet $L$-function and $\mu(n)$ is the M\"{o}bius function.
\end{theorem}

The case $m=1$ is a classical result by Kronecker (see~\cite{DukeImamogluToth2018}), while the case $m>1$ was proved by Duke--Imamo\={g}lu--T\'{o}th. Their proof involves somewhat intricate calculations by Ibukiyama--Saito~\cite[Section 2]{IbukiyamaSaito2012}. Recently, Kaneko--Mizuno~\cite[Theorem 1]{KanekoMizuno2020} presented a claim that seems to be equivalent to the above formula, although their proof also requires highly complex case distinctions. Notably, their results apply not only to the CM values but also yield similar results for the twisted trace of cycle integrals in the case where $dD>0$. Furthermore, Ibukiyama~\cite{Ibukiyama2023} has provided a more recent alternative proof using an adelic approach. In comparison, the following proof using \cref{cor:Eisen-case} is more concise. Additionally, it generalizes the proof of a special case using Hecke equivariance from earlier work of the first three authors~\cite{JeonKangKim2024}. However, it is important to emphasize that the additional condition $(D,m) = 1$ is necessary in places where results regarding Hecke equivariance are applied. We provide an alternative proof of \cref{thm:DIT} under the assumption that $(m,D) = 1$.

\begin{proof}
First, we can show that the function
\[
	M_{d}(m,s) \coloneqq \sum_{n \mid m} \mu \left(\frac{m}{n}\right) \left(\frac{d}{m/n}\right) n^{1-s} \sigma_{2s-1}(n)
\]
satisfies the following.
\begin{align}\label{eq:M-prop}
	\begin{cases}
		M_{D'\ell^2}(m,s) = M_{D'}(m,s) &\text{if } (\ell, m) = 1,\\
		M_{D'}(mn,s) = M_{D'}(m,s) M_{D'}(n,s) &\text{if } (m,n)=1,\\
		M_{D'}(p^{r+1},s) = (p^{1-s} + p^s) M_{D'}(p^r,s) - p M_{D'}(p^{r-1},s) &\text{if $p$ is prime}.
	\end{cases}
\end{align}
The first and second claims immediately follows from the multiplicativity of the functions in the definition. The third claim can be easily verified by using
\[
	M_{D'}(p^r,s) = p^{r(1-s)} \frac{1 - p^{(r+1)(2s-1)}}{1-p^{2s-1}} - \left(\frac{D'}{p}\right) p^{(r-1)(1-s)} \frac{1-p^{r(2s-1)}}{1-p^{2s-1}}.
\]

For a negative discriminant $d$, let $f_d(\tau)$ denote the weakly holomorphic modular form of weight $1/2$ on $\Gamma_0(4)$ that satisfies Kohnen's plus condition, given by the form 
\[
	f_d(\tau) = q^d + \sum_{0 < n \equiv 0, 1\ (4)} c_d(n) q^n,
\] 
as studied by Zagier~\cite[Section 5]{Zagier2002}. Moreover, the image of this function under the generalized Borcherds product $\Psi_D$ has been studied by Bruinier--Ono~\cite[Section 8]{BruinierOno2010}, and it is known that
\[
	\Psi_D(f_d) = \prod_{n=1}^\infty \prod_{b\ (D)} \left(1 - e\left(n\tau + \frac{b}{D}\right)\right)^{\left(\frac{D}{b}\right) c_d(Dn^2)} = \prod_{Q \in \mathcal{Q}_{dD}/\SL_2(\Z)} (j(\tau) - j(\alpha_Q))^{\frac{\chi_D(Q)}{w_Q}}
\]
is a weight 0 meromorphic modular form with $\Psi_D(f_d)(i\infty) = 1$, making \cref{cor:Eisen-case} applicable. Therefore, we have
\begin{align}\label{eq:Tr-Borcherds}
	\mathrm{Tr}_{d,D}(E_{1,0}(\cdot,s)) = \sum_{z \in \SL_2(\Z) \backslash \bbH} \frac{\mathrm{ord}_z(\Psi_D(f_d))}{\omega_z} E_{1,0}(z,s) = -\frac{s(1-s)}{2\pi} \langle E_{1,0}(\cdot, s), \log |\Psi_D(f_d)| \rangle^{\mathrm{reg}}.
\end{align}

For a prime $p$, let $T(p^2)$ be the Hecke operator acting on the space of weight $1/2$ modular forms, (see~\cite[Chapter IV, Section 3]{Koblitz1993}), given by
\[
	p T(p^2) \left(\sum_{n \gg -\infty} c(n) q^n\right) = \sum_{n \gg -\infty} \left(pc(p^2n) + \left(\frac{n}{p}\right) c(n) + c \left(\frac{n}{p^2}\right)\right) q^n.
\]
As shown by Zagier~\cite{Zagier2002}, by comparing the principal parts, it can be seen that
\[
	pT(p^2) f_d = pf_{d/p^2} + \left(\frac{d}{p}\right) f_d + f_{dp^2},
\]
where $f_{d/p^2}(\tau) = 0$ if $d \not\equiv 0 \pmod{p^2}$. This implies the equation concerning the order of the Borcherds products:
\[
	\mathrm{ord}_z(\Psi_D(f_{dp^2})) = \mathrm{ord}_z(\Psi_D(pT(p^2) f_d)) - \left(\frac{d}{p}\right) \mathrm{ord}_z(\Psi_D(f_d)) - p \cdot \mathrm{ord}_z (\Psi_D(f_{d/p^2})),
\]
where we use the property $\Psi_D(f+g) = \Psi_D(f) \Psi_D(g)$ that follows directly from the definition (see~\cite{BruinierOno2010}). Combining it with \eqref{eq:Tr-Borcherds} leads to
\begin{align}\label{eq:Tr-Eisen-middle}
\begin{split}
	&\mathrm{Tr}_{dp^2, D}(E_{1,0}(\cdot, s))\\
	&= -\frac{s(1-s)}{2\pi} \langle E_{1,0}(\cdot, s), \log |\Psi_D(pT(p^2) f_d)| \rangle^{\mathrm{reg}} - \left(\frac{d}{p}\right) \mathrm{Tr}_{d,D}(E_{1,0}(\cdot,s)) - p \mathrm{Tr}_{d/p^2,D}(E_{1,0}(\cdot,s)).
\end{split}
\end{align}

To compute the first term on the right-hand side of \eqref{eq:Tr-Eisen-middle}, we recall Guerzhoy's result~\cite{Guerzhoy2006} on Hecke equivariance and the generalization by the first three authors~\cite{JeonKangKim2023}, which we use here. For $f(\tau)$ with leading coefficient 1, the multiplicative Hecke operator $\mathcal{T}(p)$ of weight 0 is defined by
\[
	\mathcal{T}(p) f(\tau) = \prod_{\substack{a, d > 0 \\ ad = p}} \prod_{b\ (d)} f \left(\frac{a\tau+b}{d}\right),
\]
compared to the usual Hecke operator $T_p$ of weight 0 defined by
\[
	T_p f(\tau) = \frac{1}{p} \sum_{\substack{a, d > 0 \\ ad = p}} \sum_{b\ (d)} f \left(\frac{a\tau+b}{d}\right).
\]
It is shown in \cite[Theorem 3.1]{JeonKangKim2023} that if $p$ does not divide $D$, then
\[
	\Psi_D(pT(p^2)f_d) = \mathcal{T}(p) \Psi_D(f_d)
\]
holds. From the definitions of Hecke operators, we can derive
\[
	\log |\Psi_D(pT(p^2)f_d)| = pT_p \log |\Psi_D(f_d)|.
\]

Therefore, the first term of \eqref{eq:Tr-Eisen-middle} is calculated as
\begin{align*}
	-\frac{s(1-s)}{2\pi} \langle E_{1,0}(\cdot, s), \log |\Psi_D(pT(p^2) f_d)| \rangle^{\mathrm{reg}} &= -\frac{s(1-s)}{2\pi} p \langle E_{1,0}(\cdot, s), T_p \log |\Psi_D(f_d)| \rangle^{\mathrm{reg}}\\
		&= -\frac{s(1-s)}{2\pi} p \langle T_p E_{1,0}(\cdot, s), \log |\Psi_D(f_d)| \rangle^{\mathrm{reg}}\\
		&= -\frac{s(1-s)}{2\pi} p (p^{-s} + p^{s-1}) \langle E_{1,0}(\cdot, s), \log |\Psi_D(f_d)| \rangle^{\mathrm{reg}}\\
		&= (p^{1-s} + p^s) \mathrm{Tr}_{d,D}(E_{1,0}(\tau,s)).
\end{align*}
Here, the second equality follows from the fact that the Hecke operator $T_p$ is self-adjoint for the regularized Petersson inner product. The third equality uses the formula 
\[
	T_p E_{1,0}(\tau,s) = \frac{1}{p} \sum_{\substack{a, d > 0 \\ ad=p}} \sum_{b\ (d)} E_{1,0} \left(\frac{a\tau+b}{d}, s \right) = (p^{-s} + p^{s-1}) E_{1,0}(\tau,s),
\]
which can be shown in the same manner as in Serre~\cite[Chapter VII, 5.5]{Serre1973}. Thus, we obtain from \eqref{eq:Tr-Eisen-middle} that
\[
	\mathrm{Tr}_{dp^2, D}(E_{1,0}(\cdot,s)) = \left(p^{1-s} + p^s - \left(\frac{d}{p}\right) \right) \mathrm{Tr}_{d,D}(E_{1,0}(\cdot,s)) - p \mathrm{Tr}_{d/p^2, D}(E_{1,0}(\cdot,s)),
\]
which implies, for $(d,p) = 1$, 
\begin{align}\label{eq:Tr-rec}
\begin{split}
	\mathrm{Tr}_{dp^{2r}, D}(E_{1,0}(\cdot,s)) &= \left(p^{r(1-s)} \sigma_{2s-1}(p^r) - \left(\frac{d}{p}\right) p^{(r-1)(1-s)} \sigma_{2s-1}(p^{r-1}) \right) \mathrm{Tr}_{d,D}(E_{1,0}(\cdot,s))\\
		&= M_d(p^r, s) \mathrm{Tr}_{d,D}(E_{1,0}(\cdot,s))
\end{split}
\end{align}
by induction and the third property in \eqref{eq:M-prop}. 

In conclusion, for $d = D' m^2$ with $(m,D) = 1$, we have
\[
	\mathrm{Tr}_{d, D}(E_{1,0}(\cdot,s)) = \mathrm{Tr}_{D',D}(E_{1,0}(\cdot,s)) M_{D'}(m,s).
\]
This can be verified as follows: For the prime factorization of $m = \prod_{j=1}^r p_j^{\epsilon_j} = p_1^{\epsilon_1} m'$, applying \eqref{eq:Tr-rec} once yields
\[
	\mathrm{Tr}_{D' m'^2 p_1^{2\epsilon_1}, D}(E_{1,0}(\cdot, s)) = M_{D'm'^2}(p_1^{\epsilon_1}, s) \mathrm{Tr}_{D'm'^2, D}(E_{1,0}(\cdot, s)).
\]
Since $(m', p_1) = 1$, by the first property in \eqref{eq:M-prop}, this equals $M_{D'}(p_1^{\epsilon_1}, s) \mathrm{Tr}_{D'm'^2, D}(E_{1,0}(\cdot, s))$. By repeatedly applying this argument, we obtain
\[
	\mathrm{Tr}_{D' m^2, D}(E_{1,0}(\cdot, s)) = \mathrm{Tr}_{D', D}(E_{1,0}(\cdot, s)) \cdot \prod_{j=1}^r M_{D'}(p_j^{\epsilon_j}, s).
\]
Finally, applying the second property in \eqref{eq:M-prop} concludes the proof.
\end{proof}

Comparing the constant term of both sides of \cref{thm:DIT} at $s=1$ leads to \cite[Corollary 1.4]{Matsusaka2019} and its generalization~\cite[Corollary 1.4]{JeonKangKim2024}, while the higher coefficients at $s=1$ provide similar results for polyharmonic Maass forms studied in~\cite{LagariasRhoades2016} and \cite{Matsusaka2020}.

\bibliographystyle{amsalpha}
\bibliography{References} 

\begin{bibdiv}
\begin{biblist}

\bib{ALR2018}{article}{
      author={Andersen, Nickolas},
      author={Lagarias, Jeffrey~C.},
      author={Rhoades, Robert~C.},
       title={Shifted polyharmonic {M}aass forms for {${\rm PSL}(2,\Bbb{Z})$}},
        date={2018},
     journal={Acta Arith.},
      volume={185},
      number={1},
       pages={39\ndash 79},
}

\bib{AsaiKanekoNinomiya1997}{article}{
      author={Asai, Tetsuya},
      author={Kaneko, Masanobu},
      author={Ninomiya, Hirohito},
       title={Zeros of certain modular functions and an application},
        date={1997},
     journal={Comment. Math. Univ. St. Paul.},
      volume={46},
      number={1},
       pages={93\ndash 101},
}

\bib{BFOR2017}{book}{
      author={Bringmann, Kathrin},
      author={Folsom, Amanda},
      author={Ono, Ken},
      author={Rolen, Larry},
       title={Harmonic {M}aass forms and mock modular forms: theory and
  applications},
      series={American Mathematical Society Colloquium Publications},
   publisher={American Mathematical Society, Providence, RI},
        date={2017},
      volume={64},
}

\bib{BringmannKane2020}{article}{
      author={Bringmann, Kathrin},
      author={Kane, Ben},
       title={An extension of {R}ohrlich's theorem to the {$j$}-function},
        date={2020},
     journal={Forum Math. Sigma},
      volume={8},
       pages={Paper No. e3, 33},
}

\bib{BKLOR2018}{article}{
      author={Bringmann, Kathrin},
      author={Kane, Ben},
      author={L\"{o}brich, Steffen},
      author={Ono, Ken},
      author={Rolen, Larry},
       title={On divisors of modular forms},
        date={2018},
     journal={Adv. Math.},
      volume={329},
       pages={541\ndash 554},
}

\bib{BruinierKohnenOno2004}{article}{
      author={Bruinier, Jan},
      author={Kohnen, Winfried},
      author={Ono, Ken},
       title={The arithmetic of the values of modular functions and the
  divisors of modular forms},
        date={2004},
     journal={Compos. Math.},
      volume={140},
      number={3},
       pages={552\ndash 566},
}

\bib{BruinierOno2010}{article}{
      author={Bruinier, Jan},
      author={Ono, Ken},
       title={Heegner divisors, {$L$}-functions and harmonic weak {M}aass
  forms},
        date={2010},
     journal={Ann. of Math. (2)},
      volume={172},
      number={3},
       pages={2135\ndash 2181},
}

\bib{Choi2010}{article}{
      author={Choi, D.},
       title={Poincar\'{e} series and the divisors of modular forms},
        date={2010},
     journal={Proc. Amer. Math. Soc.},
      volume={138},
      number={10},
       pages={3393\ndash 3403},
}

\bib{CJS2023}{article}{
      author={Cogdell, James~W.},
      author={Jorgenson, Jay},
      author={Smajlovi\'{c}, Lejla},
       title={Kronecker limit functions and an extension of the
  {R}ohrlich-{J}ensen formula},
        date={2023},
     journal={Nagoya Math. J.},
      volume={252},
       pages={810\ndash 841},
}

\bib{CohenStromberg2017}{book}{
      author={Cohen, Henri},
      author={Str\"{o}mberg, Fredrik},
       title={Modular forms},
      series={Graduate Studies in Mathematics},
   publisher={American Mathematical Society, Providence, RI},
        date={2017},
      volume={179},
        note={A classical approach},
}

\bib{DukeImamogluToth2011}{article}{
      author={Duke, W.},
      author={Imamo\={g}lu, \"{O}.},
      author={T\'{o}th, \'{A}.},
       title={Cycle integrals of the {$j$}-function and mock modular forms},
        date={2011},
     journal={Ann. of Math. (2)},
      volume={173},
      number={2},
       pages={947\ndash 981},
}

\bib{DukeImamogluToth2018}{article}{
      author={Duke, W.},
      author={Imamo\={g}lu, \"{O}.},
      author={T\'{o}th, \'{A}.},
       title={Kronecker's first limit formula, revisited},
        date={2018},
     journal={Res. Math. Sci.},
      volume={5},
      number={2},
       pages={Paper No. 20, 21},
}

\bib{Erdelyi1953}{book}{
      author={Erd\'{e}lyi, Arthur},
      author={Magnus, Wilhelm},
      author={Oberhettinger, Fritz},
      author={Tricomi, Francesco~G.},
       title={Higher transcendental functions. {V}ol. {I}},
   publisher={McGraw-Hill Book Co., Inc., New York-Toronto-London},
        date={1953},
        note={Based, in part, on notes left by Harry Bateman},
}

\bib{GR2015}{book}{
      author={Gradshteyn, I.~S.},
      author={Ryzhik, I.~M.},
       title={Table of integrals, series, and products},
     edition={Eighth},
   publisher={Elsevier/Academic Press, Amsterdam},
        date={2015},
        note={Translated from the Russian, Translation edited and with a
  preface by Daniel Zwillinger and Victor Moll, Revised from the seventh
  edition [MR2360010]},
}

\bib{Guerzhoy2006}{article}{
      author={Guerzhoy, P.},
       title={On the {H}ecke equivariance of the {B}orcherds isomorphism},
        date={2006},
     journal={Bull. London Math. Soc.},
      volume={38},
      number={1},
       pages={93\ndash 96},
}

\bib{Ibukiyama2023}{misc}{
      author={Ibukiyama, Tomoyoshi},
       title={Genus character ${L}$-functions of quadratic orders in an adelic
  way and maximal orders of matrix algebras},
        date={2023},
         url={https://arxiv.org/abs/2303.14983},
        note={arXiv:2303.14983},
}

\bib{IbukiyamaSaito2012}{article}{
      author={Ibukiyama, Tomoyoshi},
      author={Saito, Hiroshi},
       title={On zeta functions associated to symmetric matrices, {II}:
  {F}unctional equations and special values},
        date={2012},
     journal={Nagoya Math. J.},
      volume={208},
       pages={265\ndash 316},
}

\bib{Iwaniec2002}{book}{
      author={Iwaniec, Henryk},
       title={Spectral methods of automorphic forms},
     edition={Second},
      series={Graduate Studies in Mathematics},
   publisher={American Mathematical Society, Providence, RI; Revista
  Matem\'{a}tica Iberoamericana, Madrid},
        date={2002},
      volume={53},
}

\bib{JeonKangKim2013}{article}{
      author={Jeon, Daeyeol},
      author={Kang, Soon-Yi},
      author={Kim, Chang~Heon},
       title={Weak {M}aass-{P}oincar\'{e} series and weight 3/2 mock modular
  forms},
        date={2013},
     journal={J. Number Theory},
      volume={133},
      number={8},
       pages={2567\ndash 2587},
}

\bib{JeonKangKim2022}{article}{
      author={Jeon, Daeyeol},
      author={Kang, Soon-Yi},
      author={Kim, Chang~Heon},
       title={On values of weakly holomorphic modular functions at divisors of
  meromorphic modular forms},
        date={2022},
     journal={J. Number Theory},
      volume={239},
       pages={183\ndash 206},
}

\bib{JeonKangKim2023}{article}{
      author={Jeon, Daeyeol},
      author={Kang, Soon-Yi},
      author={Kim, Chang~Heon},
       title={Hecke equivariance of generalized {B}orcherds products of type
  {$O(2,1)$}},
        date={2023},
     journal={J. Math. Anal. Appl.},
      volume={527},
      number={1, part 2},
       pages={Paper No. 127386, 16},
}

\bib{JKK2023Ramanujan}{article}{
      author={Jeon, Daeyeol},
      author={Kang, Soon-Yi},
      author={Kim, Chang~Heon},
       title={The {H}ecke system of harmonic {M}aass functions and applications
  to modular curves of higher genera},
        date={2023},
     journal={Ramanujan J.},
      volume={62},
      number={3},
       pages={675\ndash 717},
}

\bib{JeonKangKim2024}{misc}{
      author={Jeon, Daeyeol},
      author={Kang, Soon-Yi},
      author={Kim, Chang~Heon},
       title={Hecke equivariance of divisor lifting with respect to
  sesquiharmonic {M}aass forms},
        date={2024},
         url={https://arxiv.org/abs/2407.21447},
        note={arXiv:2407.21447},
}

\bib{KanekoMizuno2020}{article}{
      author={Kaneko, Masanobu},
      author={Mizuno, Yoshinori},
       title={Genus character {$L$}-functions of quadratic orders and class
  numbers},
        date={2020},
     journal={J. Lond. Math. Soc. (2)},
      volume={102},
      number={1},
       pages={69\ndash 98},
}

\bib{Kang2021}{article}{
      author={Kang, Soon-Yi},
       title={Divisibility properties of the {F}ourier coefficients of (mock)
  modular functions and {R}amanujan},
        date={2021},
     journal={Int. J. Number Theory},
      volume={17},
      number={3},
       pages={577\ndash 589},
}

\bib{Koblitz1993}{book}{
      author={Koblitz, Neal},
       title={Introduction to elliptic curves and modular forms},
     edition={Second},
      series={Graduate Texts in Mathematics},
   publisher={Springer-Verlag, New York},
        date={1993},
      volume={97},
}

\bib{LagariasRhoades2016}{article}{
      author={Lagarias, Jeffrey~C.},
      author={Rhoades, Robert~C.},
       title={Polyharmonic {M}aass forms for {$\mathrm{PSL}(2,\mathbb{Z})$}},
        date={2016},
     journal={Ramanujan J.},
      volume={41},
      number={1-3},
       pages={191\ndash 232},
}

\bib{Lobrich2018}{article}{
      author={L\"{o}brich, Steffen},
       title={Niebur-{P}oincar\'{e} series and traces of singular moduli},
        date={2018},
     journal={J. Math. Anal. Appl.},
      volume={465},
      number={1},
       pages={673\ndash 689},
}

\bib{MOS1966}{book}{
      author={Magnus, Wilhelm},
      author={Oberhettinger, Fritz},
      author={Soni, Raj~Pal},
       title={Formulas and theorems for the special functions of mathematical
  physics},
     edition={enlarged},
      series={Die Grundlehren der mathematischen Wissenschaften, Band 52},
   publisher={Springer-Verlag New York, Inc., New York},
        date={1966},
}

\bib{Matsusaka2019}{article}{
      author={Matsusaka, Toshiki},
       title={Traces of {CM} values and cycle integrals of polyharmonic {M}aass
  forms},
        date={2019},
     journal={Res. Number Theory},
      volume={5},
      number={1},
       pages={Paper No. 8, 25},
}

\bib{Matsusaka2020}{article}{
      author={Matsusaka, Toshiki},
       title={Polyharmonic weak {M}aass forms of higher depth for {${\rm
  SL}_2(\Bbb Z)$}},
        date={2020},
     journal={Ramanujan J.},
      volume={51},
      number={1},
       pages={19\ndash 42},
}

\bib{Niebur1973}{article}{
      author={Niebur, Douglas},
       title={A class of nonanalytic automorphic functions},
        date={1973},
     journal={Nagoya Math. J.},
      volume={52},
       pages={133\ndash 145},
}

\bib{NIST}{misc}{
      author={NIST},
       title={Digital {L}ibrary of {M}athematical {F}unctions},
        note={\url{https://dlmf.nist.gov}},
}

\bib{Rohrlich1984}{article}{
      author={Rohrlich, David~E.},
       title={A modular version of {J}ensen's formula},
        date={1984},
     journal={Math. Proc. Cambridge Philos. Soc.},
      volume={95},
      number={1},
       pages={15\ndash 20},
}

\bib{Serre1973}{book}{
      author={Serre, J.-P.},
       title={A course in arithmetic},
      series={Graduate Texts in Mathematics, No. 7},
   publisher={Springer-Verlag, New York-Heidelberg},
        date={1973},
        note={Translated from the French},
}

\bib{Yang2006}{article}{
      author={Yang, Yifan},
       title={Defining equations of modular curves},
        date={2006},
     journal={Adv. Math.},
      volume={204},
      number={2},
       pages={481\ndash 508},
}

\bib{Zagier2002}{incollection}{
      author={Zagier, Don},
       title={Traces of singular moduli},
        date={2002},
   booktitle={Motives, polylogarithms and {H}odge theory, {P}art {I} ({I}rvine,
  {CA}, 1998)},
      series={Int. Press Lect. Ser.},
      volume={3},
   publisher={Int. Press, Somerville, MA},
       pages={211\ndash 244},
}

\end{biblist}
\end{bibdiv}

\end{document}